\documentclass{article}
\usepackage{amssymb}
\usepackage{amsfonts}
\usepackage{amsmath}

\newtheorem{theorem}{Theorem}

\newtheorem{corollary}[theorem]{Corollary}

\newtheorem{definition}[theorem]{Definition}
\newtheorem{example}[theorem]{Example}

\newtheorem{lemma}[theorem]{Lemma}
\newtheorem{notation}[theorem]{Notation}

\newtheorem{remark}[theorem]{Remark}

\newenvironment{proof}[1][Proof]{\noindent\textbf{#1.} }{\ \rule{0.5em}{0.5em}}
\input{tcilatex}
\begin{document}

\title{The partial sum process of orthogonal expansions as geometric rough
process with Fourier series as an example---an improvement of
Menshov-Rademacher theorem}
\author{Terry J. Lyons\thanks{%
University of Oxford and Oxford-man Institute, Email:
terry.lyons@maths.ox.ac.uk} \ \ \ \ Danyu Yang\thanks{%
University of Oxford and Oxford-man Institute, Email: yangd@maths.ox.ac.uk}}
\maketitle

\begin{abstract}
The partial sum process of orthogonal expansion $\sum_{n\geq 0}c_{n}u_{n}$
is a geometric $2$-rough process, for any orthonormal system $\left\{
u_{n}\right\} _{n\geq 0}$ in $L^{2}$ and any sequence of numbers $\left\{
c_{n}\right\} $ satisfying $\sum_{n\geq 0}\left( \log _{2}\left( n+1\right)
\right) ^{2}\left\vert c_{n}\right\vert ^{2}<\infty $. Since being a
geometric $2$-rough process implies the existence of a limit function up to
a null set, our theorem could be treated as an improvement of
Menshov-Rademacher theorem. For Fourier series, the condition can be
strengthened to $\sum_{n\geq 0}\log _{2}\left( n+1\right) \left\vert
c_{n}\right\vert ^{2}<\infty $, which is equivalent to $\int_{-\pi }^{\pi
}\int_{-\pi }^{\pi }\frac{\left\vert f\left( u\right) -f\left( v\right)
\right\vert ^{2}}{\left\vert \sin \frac{u-v}{2}\right\vert }dudv<\infty $
(with $f$ the limit function).
\end{abstract}


\section{Introduction}

\begin{definition}
\bigskip $\left\{ u_{n}\right\} _{n=0}^{\infty }\ $is said to be an
orthonormal system in $L^{2}$ and denoted as $\left\{ u_{n}\right\} \in
L^{2} $, if there exist measure space $\left( \Omega ,\mathcal{F},\mu
\right) $ and Hilbert space $\left( \mathcal{V},\left\langle \cdot ,\cdot
\right\rangle \right) $, such that $u_{n}:\left( \Omega ,\mathcal{F},\mu
\right) \rightarrow \left( \mathcal{V},\left\langle \cdot ,\cdot
\right\rangle \right) $, $\forall n\in 
\mathbb{N}
$, and $\int_{\Omega }\left\langle u_{n}\left( \omega \right) ,u_{m}\left(
\omega \right) \right\rangle \mu \left( d\omega \right) =\delta _{mn}$, $%
\forall n,m\in 
\mathbb{N}
$.
\end{definition}

\begin{definition}
Suppose $\left\{ u_{n}\right\} _{n=0}^{\infty }$ is an orthonormal system in 
$L^{2}$, and $\left\{ c_{n}\right\} _{n=0}^{\infty }$ is a sequence of
numbers. Then the partial sum process $X$ of $\sum_{k=0}^{\infty }c_{n}u_{n}$
is a process indexed by $%
\mathbb{N}
$, got by defining for each $\omega \in \Omega $, 
\begin{equation}
X_{n}\left( \omega \right) :=\sum_{k=0}^{n}c_{k}u_{k}\left( \omega \right) ,%
\text{ }\forall n\in 
\mathbb{N}
\text{.}  \label{Definition of X theta}
\end{equation}
\end{definition}

We will identify a condition on $\left\{ c_{n}\right\} $, under which $X$ is
a rough path with finite $2$-variation on the half line, almost everywhere
on $\Omega $ and for every choice of orthonormal system. Since almost
everywhere finiteness of $2$-variation of partial sum process implies the
existence of a limit function upto a null set, our topic has a direct
connection with a.e. convergence of general orthonormal series, which dates
back to Weyl\cite{H. Weyl}.

\begin{definition}[Weyl multiplier for property $p$]
Suppose $\left\{ w\left( n\right) \right\} _{n=0}^{\infty }$ is a sequence
of positive non-decreasing numbers. $\left\{ w\left( n\right) \right\} $ is
said to be a Weyl multiplier for property $p$, if $p$ holds for all
orthogonal series $\sum_{n=0}^{\infty }c_{n}u_{n}$, for any orthonormal
system $\left\{ u_{n}\right\} $ in $L^{2}$ and any sequence of numbers $%
\left\{ c_{n}\right\} $ satisfying $\sum_{n=0}^{\infty }w\left( n\right)
\left\vert c_{n}\right\vert ^{2}<\infty $.
\end{definition}

Not every orthogonal series with coefficients in $l^{2}$ is convergent.
There exists an $L^{2}$ Fourier series which diverges a.e. after some
rearrangement, \cite{Z. Zahorski}. In fact, for any complete orthonormal
system in $L^{2}\left( \left( 0,1\right) ,%
\mathbb{R}
\right) $, there exists a continuous function, whose expansion diverges
unboundedly almost everywhere after some rearrangement, \cite{A. M. Olevskii}%
. Moreover, Banach \cite{S. Banach} proved that, if we equip sequences in $%
L^{2}\left( \left( 0,1\right) ,%
\mathbb{R}
\right) $ with the metric%
\begin{equation}
d\left( \left\{ u_{n}\right\} ,\left\{ v_{n}\right\} \right)
=\sum_{n=0}^{\infty }\frac{1}{2^{n}}\frac{\left\Vert u_{n}-v_{n}\right\Vert
_{L^{2}}}{1+\left\Vert u_{n}-v_{n}\right\Vert _{L^{2}}}\text{, }\left\Vert
u\right\Vert _{L^{2}}=\left( \int_{0}^{1}u^{2}\left( x\right) dx\right) ^{%
\frac{1}{2}},  \label{Definition of a metric on orthonormal systems}
\end{equation}%
then the set of orthonormal systems, whose expansions of all bounded
variation functions diverge unboundedly almost everywhere, is a $G_{\delta }$
and everywhere second category subset of sequences in $L^{2}\left( \left(
0,1\right) ,%
\mathbb{R}
\right) $.

The exact Weyl multiplier for almost everywhere convergence of general
orthogonal series is found by Menshov\cite{D. Menshov}\ and Rademacher\cite%
{H. Rademacher}.

\begin{theorem}[Menshov-Rademacher]
\label{Menshov-Rademacher Theorem}The orthogonal series $\sum_{n=0}^{\infty
}c_{n}u_{n}$ converges almost everywhere, for any $\left\{ u_{n}\right\}
_{n=0}^{\infty }\in L^{2}$ and any sequence of numbers $\left\{
c_{n}\right\} _{n=0}^{\infty }$ satisfying%
\begin{equation}
\sum_{n=0}^{\infty }\left( \log _{2}\left( n+1\right) \right) ^{2}\left\vert
c_{n}\right\vert ^{2}<\infty .  \label{Menchov-Rademacher condition}
\end{equation}%
Furthermore, $\left( \log _{2}\left( n+1\right) \right) ^{2}$ in $\left( \ref%
{Menchov-Rademacher condition}\right) $ can not be replaced by $o\left(
\left( \log _{2}\left( n+1\right) \right) ^{2}\right) $,\ and there exists
an absolute constant $C$ such that%
\begin{equation}
\int_{\Omega }\max_{0\leq i\leq j<\infty }|\hspace{-0.01in}%
|\sum_{n=i}^{j}c_{n}u_{n}\left( \omega \right) |\hspace{-0.01in}|^{2}\mu
\left( d\omega \right) \leq C\sum_{n=0}^{\infty }\left( \log _{2}\left(
n+1\right) \right) ^{2}\left\vert c_{n}\right\vert ^{2}.
\label{Menshov-Rademacher inequality}
\end{equation}
\end{theorem}

Although its estimation is rough using Cauchy--Schwarz inequality (p251\cite%
{B. S. Kashin and A. A. Saakyan}), the Weyl multiplier $\{\left( \log
_{2}\left( n+1\right) \right) ^{2}\}$ is exact: For any Weyl multiplier $%
\left\{ w\left( n\right) \right\} $ satisfying $w\left( n\right) =o(\left(
\log _{2}\left( n+1\right) \right) ^{2})$, there exists an a.e. divergent
orthogonal series $\sum_{n}c_{n}u_{n}$, whose coefficients satisfy $%
\sum_{n}w\left( n\right) \left\vert c_{n}\right\vert ^{2}<\infty $ (p254\cite%
{B. S. Kashin and A. A. Saakyan}). (The main idea is to glue independent
pieces of finite orthogonal sequences together, where each piece provides a
constant increment on a sufficiently large set, then almost everywhere
divergence follows from Borel-Cantelli lemma.{\small )}

Moreover, as a remarkable improvement of the above counter-examples, Tandori%
\cite{K. Tandori} showed that: if the absolute value of $c_{n}$ is monotone
decreasing and $\sum_{n}\left( \log _{2}\left( n+1\right) \right)
^{2}\left\vert c_{n}\right\vert ^{2}=\infty $, then there exists $\left\{
u_{n}\right\} \in L^{2}\ $such that $\sum_{n}c_{n}u_{n}$ diverges a.e..
Thus, if the absolute value of $\left\{ c_{n}\right\} $ is monotone
decreasing, then the necessary and sufficient condition for $%
\sum_{n}c_{n}u_{n}$ to converge almost everywhere for all $\left\{
u_{n}\right\} \in L^{2}$ is $\sum_{n}\left( \log _{2}\left( n+1\right)
\right) ^{2}\left\vert c_{n}\right\vert ^{2}<\infty $.

A recent improvement of Menshov-Rademacher Theorem by A. Lewko and M. Lewko 
\cite{L. Allison} strengthened a.e. finite $\infty $-variation to a.e.
finite $2$-variation. They decompose the partial sum process into the sum of
two, one of which encodes long range displacement, while the other keeps
returning to origin. The power of this decomposition already manifested
itself in the proof of Menshov-Rademacher theorem. We will use this
decomposition, and show that the partial sum process is a geometric rough
process.

For a specific orthonormal system, Weyl multiplier for a.e. convergence can
be strengthened, even $w\left( n\right) =1$ for all $n$. In that case, the
orthonormal system is called a convergent system. Among those convergent
systems, almost everywhere convergence of Fourier series came as a deep
theorem by Carleson\cite{L. Carleson}. Hunt\cite{R. Hunt} extended
Carleson's result to $L_{r}$, $1<r<\infty $, and proved: 
\begin{equation}
\left( \int_{-\pi }^{\pi }\left\Vert X\left( \theta \right) \right\Vert
_{\infty -var}^{r}d\theta \right) ^{\frac{1}{r}}\leq C_{r}\left( \int_{-\pi
}^{\pi }\left\vert f\left( \theta \right) \right\vert ^{r}d\theta \right) ^{%
\frac{1}{r}},  \label{Hunt's inequality}
\end{equation}%
where $X\left( \theta \right) $ is the partial sum process of Fourier series
of $f$ at $\theta $. Moreover, in a recent paper by Oberlin, Seeger, Tao,
Thiele and Wright\cite{R. Oberlin}, they proved a $p$-variation version of
Carleson's theorem, which is a deep result and mainly the inequality: when $%
r>1$ and $p>\max \left\{ 2,r/\left( r-1\right) \right\} $,%
\begin{equation*}
\left( \int_{-\pi }^{\pi }\left\Vert X\left( \theta \right) \right\Vert
_{p-var}^{r}d\theta \right) ^{\frac{1}{r}}\leq C_{p,r}\left( \int_{-\pi
}^{\pi }\left\vert f\left( \theta \right) \right\vert ^{r}d\theta \right) ^{%
\frac{1}{r}}.
\end{equation*}%
Thus, the partial sum process of $L^{2}$ Fourier series has finite $p$%
-variation a.e., for any $p>2$. As a complement to \cite{R. Oberlin}, in 
\cite{L. Allison}, the authors proved that $\left\{ \log _{2}\left(
n+1\right) \right\} $ is a Weyl multiplier for a.e. finite $2$-variation of
partial sum process of Fourier series.

We strengthen Menshov-Rademacher theorem by identifying $\{\left( \log
_{2}\left( n+1\right) \right) ^{2}\}$ as the exact Weyl multiplier for the
partial sum process to be a geometric $2$-rough process, and for Fourier
series, the Weyl multiplier can be improved to $\left\{ \left( \log
_{2}\left( n+1\right) \right) \right\} $.

\section{Geometric $2$-rough path}

Before proceeding to our proofs, we clarify the definition of geometric $2$%
-rough path on $%
\mathbb{N}
$, following \cite{T. J. Lyons notes} with small modifications. (Rough paths
on $%
\mathbb{N}
$ is just a reparametrisation of piecewise-linear rough paths on $\left[ 0,1%
\right] $.)

\begin{notation}
Denote $%
\mathbb{N}
:=\left\{ 0,1,\dots \right\} $ and $\bigtriangleup _{%
\mathbb{N}
}:=\left\{ \left( i,j\right) |i\leq j,i\in 
\mathbb{N}
,j\in 
\mathbb{N}
\right\} $.
\end{notation}

\begin{definition}[$p$-variation]
Suppose $\left( \mathcal{V},\left\Vert \cdot \right\Vert \right) $ is a
Banach space, and $\alpha :\bigtriangleup _{%
\mathbb{N}
}\rightarrow \mathcal{V}$ satisfying $\alpha \left( k,k\right) =0$, $\forall
k\in 
\mathbb{N}
$. Then for $p\in \lbrack 1,\infty )$, define the $p$-variation of $\alpha $
as 
\begin{equation*}
\left\Vert \alpha \right\Vert _{p-var}:=\sup_{N\geq 1}\left( \sup_{0\leq
k_{0}<\cdots <k_{n}\leq N}\sum_{j=0}^{n-1}\left\Vert \alpha \left(
k_{j},k_{j+1}\right) \right\Vert ^{p}\right) ^{\frac{1}{p}}.
\end{equation*}
\end{definition}

For fixed $\alpha $, the function $p\mapsto \left\Vert \alpha \right\Vert
_{p-var}$ is non-increasing on $p\in \left[ 1,\infty \right] $, so $%
\left\Vert \alpha \right\Vert _{\infty -var}\leq \left\Vert \alpha
\right\Vert _{q-var}\leq \left\Vert \alpha \right\Vert _{p-var}$ for $1\leq
p\leq q\leq \infty $. The function $\gamma :%
\mathbb{N}
\rightarrow \mathcal{V}$ can be treated as a function $\widetilde{\gamma }$
on $\bigtriangleup _{%
\mathbb{N}
}$ by setting $\widetilde{\gamma }\left( k_{1},k_{2}\right) :=\gamma \left(
k_{1}\right) -\gamma \left( k_{2}\right) $, $\forall \left(
k_{1},k_{2}\right) \in \bigtriangleup _{%
\mathbb{N}
}$.

\begin{notation}
Suppose $\left( \mathcal{V},\left\Vert \cdot \right\Vert \right) $ is a
Banach space. For $u,v\in \mathcal{V}$, denote $\left[ u,v\right] :=u\otimes
v-v\otimes u$, with $\otimes $ the tensor product.
\end{notation}

We assume the norm on tensor product satisfies (upto an universal constant)%
\begin{equation}
\left\Vert u\otimes v\right\Vert \leq \left\Vert u\right\Vert \left\Vert
v\right\Vert \text{, \ }\forall u,v\in \mathcal{V}.
\label{Norm on tensor product}
\end{equation}%
Property $\left( \ref{Norm on tensor product}\right) $ holds e.g. when $%
\mathcal{V}$ is finite dimensional or when $\mathcal{V}^{\otimes 2}$ is
equipped with projective/injective tensor norm (Prop 2.1 and Prop 3.1 \cite%
{R.A.Ryan}).

\begin{notation}
Denote $\mathcal{V}^{\otimes 2}$ as the completion of $\left\{
\sum_{i=1}^{n}u_{i}\otimes v_{i}|u_{i},v_{i}\in \mathcal{V},n\geq 1\right\} $
w.r.t. the norm selected (which satisfies $\left( \ref{Norm on tensor
product}\right) $).
\end{notation}

\begin{definition}
\label{Definition of area}Suppose $\gamma :%
\mathbb{N}
\rightarrow \mathcal{V}$. Then we define the area of $\gamma $, $A\left(
\gamma \right) :\bigtriangleup _{%
\mathbb{N}
}\rightarrow \mathcal{V}^{\otimes 2}$ by setting, 
\begin{equation*}
A\left( \gamma \right) \left( k_{1},k_{2}\right) =0\text{, when }k_{2}=k_{1}%
\text{ or }k_{1}+1\text{,}
\end{equation*}%
and when $k_{2}\geq k_{1}+2$, 
\begin{equation*}
A\left( \gamma \right) \left( k_{1},k_{2}\right) :=2^{-1}\sum_{k_{1}\leq
j_{1}<j_{2}\leq k_{2}-1}\left[ \gamma \left( j_{1}+1\right) -\gamma \left(
j_{1}\right) ,\gamma \left( j_{2}+1\right) -\gamma \left( j_{2}\right) %
\right] \text{.}
\end{equation*}
\end{definition}

Then it can be verified that, for any $0\leq k_{1}\leq k_{2}\leq
k_{3}<\infty $,%
\begin{eqnarray}
A\left( \gamma \right) \left( k_{1},k_{3}\right) &=&A\left( \gamma \right)
\left( k_{1},k_{2}\right) +A\left( \gamma \right) \left( k_{2},k_{3}\right)
\label{definition of multiplicative functional} \\
&&+\frac{1}{2}\left[ \gamma \left( k_{2}\right) -\gamma \left( k_{1}\right)
,\gamma \left( k_{3}\right) -\gamma \left( k_{2}\right) \right] ,  \notag
\end{eqnarray}%
which is called multiplicativity of $\left( \gamma ,A\left( \gamma \right)
\right) $.

\begin{notation}[$G^{\left( 2\right) }$ norm]
Suppose $\Gamma :\bigtriangleup _{%
\mathbb{N}
}\rightarrow \mathcal{V}\oplus \mathcal{V}^{\otimes 2}=\left( \gamma ,\alpha
\right) $. Define the $2$-rough norm $\left\Vert \cdot \right\Vert
_{G^{\left( 2\right) }}$ of $\Gamma $ as%
\begin{equation*}
\left\Vert \Gamma \right\Vert _{G^{\left( 2\right) }}:=\left( \left\Vert
\gamma \right\Vert _{2-var}^{2}+\left\Vert \alpha \right\Vert
_{1-var}\right) ^{\frac{1}{2}}.
\end{equation*}
\end{notation}

\begin{definition}[geometric $2$-rough path]
Suppose $\gamma :%
\mathbb{N}
\rightarrow \mathcal{V}$. Then $\left( \gamma ,A\left( \gamma \right)
\right) $ is called a geometric $2$-rough path, if $\left\Vert \left( \gamma
,A\left( \gamma \right) \right) \right\Vert _{G^{\left( 2\right) }}<\infty $.
\end{definition}

The original definition in \cite{T. J. Lyons notes} of $\Gamma =\left(
\gamma ,\alpha \right) $ being a geometric $2$-rough path is that $\Gamma $
can be approximated by a sequence of bounded variation paths (and their
area) in $\left\Vert \cdot \right\Vert _{G^{\left( 2\right) }}$ norm. Here
when $\left\Vert \left( \gamma ,A\left( \gamma \right) \right) \right\Vert
_{G^{\left( 2\right) }}<\infty $, the truncation of $\gamma $ on $0,\dots ,n$
will function as the bounded variation paths.

\begin{definition}[area process]
\label{Definition of area process}Suppose $X$ is a process defined on $%
\left( \Omega ,\mathcal{F},\mu \right) $ indexed by $%
\mathbb{N}
$. Define the area process of $X$ as $\left( A\left( X\right) \right) \left(
\omega \right) :=A\left( X\left( \omega \right) \right) $, $\forall \omega
\in \Omega $.
\end{definition}

\begin{definition}[geometric $2$-rough process]
$\left( X,A\left( X\right) \right) $ is called a geometric $2$-rough process
if $\left( X\left( \omega \right) ,\left( A\left( X\right) \right) \left(
\omega \right) \right) $ is a geometric $2$-rough path for almost every $%
\omega $.
\end{definition}

The entry of area is very natural. Suppose $\gamma :\left[ 0,T\right]
\rightarrow \mathcal{V}$ is a path of finite $p$-variation. Consider the
following differential equation:%
\begin{equation*}
d\alpha _{\gamma }\left( t\right) =\left( \gamma \left( t\right) -\gamma
\left( 0\right) \right) \otimes d\gamma \left( t\right) \text{, \ }\alpha
_{\gamma }\left( 0\right) =\xi .
\end{equation*}%
According to Young's integral \cite{L. C. Young}, when $1\leq p<2$, $\alpha
_{\gamma }$ can be defined through Riemann sums, and is continuous in
p-variation w.r.t. $\gamma $: 
\begin{multline}
\left\Vert \alpha _{\gamma _{1}}-\alpha _{\gamma _{2}}\right\Vert _{p-var, 
\left[ 0,T\right] }  \notag \\
\leq C_{p}\left( \left\Vert \gamma _{1}\right\Vert _{p-var,\left[ 0,T\right]
}+\left\Vert \gamma _{2}\right\Vert _{p-var,\left[ 0,T\right] }\right)
\left\Vert \gamma _{1}-\gamma _{2}\right\Vert _{p-var,\left[ 0,T\right] }.
\end{multline}%
However, this is no longer true when $p=2$. Actually, when equipping the
space of smooth paths with $2$-variation norm, the path$\mapsto $area
operator is not continuous, nor bounded, and (when area equipped with $q$%
-variation for $q>1$) not closable \cite{Yang}. On the other hand, if a path
can be enhanced into a geometric $2$-rough path, rough path theory gives
meaning to differential equations driven by enhanced $\gamma $, and the
solution is continuous in rough path norm w.r.t. the driving rough path.
However, such lift does not always exist, \cite{N. Victor} and \cite{Yang}.
(For more systematical treatments of rough path, please refer to \cite{T. J.
Lyons notes}, \cite{T. J. Lyons and Z. Qian} and \cite{P. Friz}.)

\section{Main Result}

Suppose $\left\{ u_{n}\right\} $ is an orthonormal system in $L^{2}$ and $%
\left\{ c_{n}\right\} $ a sequence of numbers. Using techniques in rough
analysis (e.g. \cite{T. J. Lyons and Z. Qian},\cite{B. M. Hambly and T. J.
Lyons},\cite{T. J. Lyons and O. Zeitouni}), we proved:

\begin{theorem}
\label{Theorem partial sum process of general orthogonal expansion as
geometric rough process}The partial sum process of $\sum_{n}c_{n}u_{n}$,
when enhanced by its area process, is a geometric $2$-rough process (denoted
as $\mathbf{X}$) for any orthonormal system $\left\{ u_{n}\right\} \in
L^{2}\ $and any $\left\{ c_{n}\right\} $ satisfying $\sum_{n=0}^{\infty
}\left( \log _{2}\left( n+1\right) \right) ^{2}\left\vert c_{n}\right\vert
^{2}<\infty $. Moreover, $\left( \log _{2}\left( n+1\right) \right) ^{2}$
can not be replaced by $o(\left( \log _{2}(n+1)\right) ^{2})$, and%
\begin{equation}
\int_{\Omega }\left\Vert \mathbf{X}\left( \omega \right) \right\Vert
_{G^{\left( 2\right) }}^{2}\mu \left( d\omega \right) \leq
121\sum_{n=0}^{\infty }\left( \log _{2}\left( n+1\right) \right)
^{2}\left\vert c_{n}\right\vert ^{2}.
\label{Bound for 2-rough path norm for general orthogonal series}
\end{equation}
\end{theorem}

Exactness of $\{(\log _{2}(n+1))^{2}\}$ follows from Menshov-Rademacher
theorem. It is an improvement of Menshov-Rademacher Theorem because $%
\left\Vert X\left( \omega \right) \right\Vert _{\infty -var}\leq \left\Vert
X\left( \omega \right) \right\Vert _{2-var}\leq \left\Vert \mathbf{X}\left(
\omega \right) \right\Vert _{G^{\left( 2\right) }},$ $\forall \omega $.

\begin{definition}
$\left\{ u_{n}\right\} \in L^{2}$ is said to have the Hardy property with
constant $C$, if for any sequence of numbers $\left\{ a_{n}\right\}
_{n=0}^{\infty }$ satisfying $\sum_{n=0}^{\infty }\left\vert
a_{n}\right\vert ^{2}<\infty $,%
\begin{equation}
\int_{\Omega }\sup_{0\leq i\leq j<\infty }\left\Vert
\sum_{k=i}^{j}a_{k}u_{k}\left( \omega \right) \right\Vert ^{2}\mu \left(
d\omega \right) \leq C\left( \sum_{n=0}^{\infty }\left\vert a_{n}\right\vert
^{2}\right) .  \label{Hardy property}
\end{equation}
\end{definition}

\begin{theorem}
\label{Theorem partial sum process of complete orthogonal expansion as
geometric rough process}Suppose $\left\{ u_{n}\right\} \in L^{2}$ has the
Hardy property with constant $C$. Then, for $\left\{ c_{n}\right\} $
satisfying $\sum_{n}\log _{2}\left( n+1\right) \left\vert c_{n}\right\vert
^{2}<\infty $, the partial sum process of $\sum_{n}c_{n}u_{n}$, when
enhanced by its area process, is a geometric $2$-rough process (denoted as $%
\mathbf{X}$). Moreover, 
\begin{equation}
\int_{\Omega }\left\Vert \mathbf{X}\left( \omega \right) \right\Vert
_{G^{\left( 2\right) }}^{2}\mu \left( d\omega \right) \leq \left(
604+26C\right) \sum_{n=0}^{\infty }\log _{2}\left( n+1\right) \left\vert
c_{n}\right\vert ^{2}.
\label{Bound for 2-rough path norm for convergent orthogonal series}
\end{equation}
\end{theorem}

Almost everywhere finiteness of $2$-variation of the partial sum process in
Theorem \ref{Theorem partial sum process of general orthogonal expansion as
geometric rough process} and Theorem \ref{Theorem partial sum process of
complete orthogonal expansion as geometric rough process} is proved in \cite%
{L. Allison}. Thus, since area vanishes if the orthonormal system is
one-dimensional, our result is an improvement only in multi-dimensional case.

\begin{corollary}
\label{Corollary for geometric rough process of Fourier series}Theorem \ref%
{Theorem partial sum process of complete orthogonal expansion as geometric
rough process} holds for Fourier system, where $\log _{2}\left( n+1\right) $
in $\left( \ref{Bound for 2-rough path norm for convergent orthogonal series}%
\right) $\ can not be replaced by $o\left( \log _{2}\left( n+1\right)
\right) $.
\end{corollary}

This corollary follows from Theorem \ref{Theorem partial sum process of
complete orthogonal expansion as geometric rough process} and
Carleson--Hunt's inequality $\left( \ref{Hunt's inequality}\right) $ (see
also \cite{C. Fefferman}). The lower bound, as indicated in \cite{R. Oberlin}
or \cite{L. Allison}, can be obtained in the case of de la Vall\'{e}%
e-Poussin kernel, or say, Dirichlet kernel.

It is reasonable to define sobolev space $H_{Log}^{s}$ for $s>0$, as the
space of functions in $L^{2}\left( \left[ -\pi ,\pi \right] ,%
\mathbb{R}
^{d}\right) $, whose Fourier coefficients satisfy%
\begin{equation*}
\sum_{n=0}^{\infty }\left( \log _{2}\left( n+1\right) \right)
^{2s}\left\vert c_{n}\right\vert ^{2}<\infty .
\end{equation*}%
Then we have the following identification of functions in $H_{Log}^{s}$
(when $s=\frac{1}{2}$, the equivalency is proved in Thm 4 \cite{W. Beckner}%
). (Euclidean norm is used in Theorem \ref{Theorem equivalent norm on
Fourier series}, so that the constants $k_{s}$ and $K_{s}$ do not depend on
dimension $d$.)

\begin{theorem}
\label{Theorem equivalent norm on Fourier series}For any $s\in \left(
0,\infty \right) $, there exist constants $0<k_{s}\leq K_{s}<\infty $, s.t.
for any $f\in L^{2}\left( \left[ -\pi ,\pi \right] ,%
\mathbb{R}
^{d}\right) $ with Fourier coefficients $\left\{ c_{n}\right\} $, \ 
\begin{gather*}
\text{if denote }L\!\left( f\right) :=\int_{-\pi }^{\pi }\int_{-\pi }^{\pi }%
\frac{\left\vert f\left( u\right) -f\left( v\right) \right\vert ^{2}}{%
\left\vert \sin \frac{u-v}{2}\right\vert }(\log _{2}\frac{\pi }{\left\vert
\sin \frac{u-v}{2}\right\vert })^{2s-1}dudv \\
\text{and }l\!\left( f\right) :=\sum_{n=0}^{\infty }\left( \log _{2}\left(
n+1\right) \right) ^{2s}\left\vert c_{n}\right\vert ^{2}\text{, then }%
k_{s}\,l\!\left( f\right) \leq L\!\left( f\right) \leq K_{s}\,l\!\left(
f\right) .
\end{gather*}
\end{theorem}

\begin{corollary}
Suppose $f:\left[ -\pi ,\pi \right] \rightarrow 
\mathbb{R}
^{d}$ satisfying $\int_{-\pi }^{\pi }\int_{-\pi }^{\pi }\frac{\left\vert
f\left( u\right) -f\left( v\right) \right\vert ^{2}}{\left\vert \sin \frac{%
u-v}{2}\right\vert }dudv<\infty $. Then $f$ is in $L^{2}$, and the partial
sum process of the Fourier series of $f$, when enhanced by its area process,
is a geometric $2$-rough process (denoted as $\mathbf{X}$). Moreover,%
\begin{equation*}
\int_{-\pi }^{\pi }\left\Vert \mathbf{X}\left( \theta \right) \right\Vert
_{G^{\left( 2\right) }}^{2}d\theta \leq \left( 604+26C_{0}\right) k_{\frac{1%
}{2}}^{-1}\int_{-\pi }^{\pi }\int_{-\pi }^{\pi }\frac{\left\vert f\left(
u\right) -f\left( v\right) \right\vert ^{2}}{\left\vert \sin \frac{u-v}{2}%
\right\vert }dudv,
\end{equation*}%
where $C_{0}$ is the Hardy constant for $L^{2}$ Fourier series and $k_{\frac{%
1}{2}}$ is defined in Theorem \ref{Theorem equivalent norm on Fourier series}%
.
\end{corollary}

This corollary follows trivially from Corollary \ref{Corollary for geometric
rough process of Fourier series} and Theorem \ref{Theorem equivalent norm on
Fourier series}.

The function $x^{-\frac{1}{2}}\left\vert \log _{2}\frac{x}{2}\right\vert
^{-\left( s+\frac{1}{2}\right) }\left\vert \log _{2}\left( 2\left\vert \log
_{2}\frac{x}{2}\right\vert \right) \right\vert ^{-\frac{1}{2}-\epsilon }$, $%
x\in \left( 0,1\right) $, (according to Theorem $2.24$ p190 Vol I \cite%
{Zygmund}) is included in $H_{Log}^{s}$ when $\epsilon >0$, while not
included in $H_{Log}^{s}$ when $\epsilon \leq 0$.

Although for Fourier series, $\log _{2}\left( n+1\right) $ in Corollary \ref%
{Corollary for geometric rough process of Fourier series} can not be
replaced by $o\left( \log _{2}\left( n+1\right) \right) $, $\sum_{n}\log
_{2}\left( n+1\right) \left\vert c_{n}\right\vert ^{2}<\infty $ is not
necessary for the partial sum process of Fourier series to be a geometric $2$%
-rough process (i.e. an almost everywhere finite random variable with
infinite expectation). In fact, we give a little stronger statement.

\begin{example}
\label{Example lnn is not sufficient for Fourier series}Suppose $\left\{
w\left( n\right) \right\} $ is a Weyl multiplier that $n\mapsto \frac{%
w\left( n\right) }{\left( \log _{2}\log _{2}n\right) ^{2}}\ $is strictly
increasing from some point on, and $\lim_{n\rightarrow \infty }\frac{w\left(
n\right) }{\left( \log _{2}\log _{2}n\right) ^{2}}=\infty $. Then there
exists an $L^{2}$ Fourier series $\sum_{n=1}^{\infty }c_{n}e^{in\theta }$,
such that its partial sum process is a geometric $2$-rough process, but $%
\sum_{n=1}^{\infty }w\left( n\right) \left\vert c_{n}\right\vert ^{2}=\infty 
$.
\end{example}

The above example is $2$-dimensional, so area is non-trivial.

One might be tempted to ask whether all $L^{2}$ Fourier series have finite $%
2 $-variation a.e., which, however, is not true. It is proved in \cite{R.
Jones} that there exists a bounded function, whose Fourier series has
infinite $2$-variation a.e.. Their proof relies on nontrivial estimates on $%
2 $-variation of partial sum process of i.i.d. sequences, \cite{J. Qian}. In
this paper, we provide a self-contained proof, where we use the upper
semi-continuity of cumulative distribution function of $p$-variation. This
example is constructed without knowledge of \cite{R. Jones}, nor the result
in \cite{J. Qian}.

\begin{example}
\label{Example of almost everywhere infinite quadratic variation}There
exists an $L^{2}$ Fourier series whose partial sum process has infinite $2$%
-variation almost everywhere.
\end{example}

\section{Proof of Theorem \protect\ref{Theorem partial sum process of
general orthogonal expansion as geometric rough process} and Theorem \protect
\ref{Theorem partial sum process of complete orthogonal expansion as
geometric rough process}}

\begin{definition}
Denote $%
\mathbb{N}
:=\left\{ 0,1,2,\dots \right\} $, and $J$ is said to be an interval, if $J=%
\left[ m,n\right] $ for $m\in 
\mathbb{N}
$, $n\in 
\mathbb{N}
$, $m<n$.
\end{definition}

\begin{definition}
$D=\left\{ \left[ k_{j},k_{j+1}\right] \right\} _{j=0}^{n}$ is said to be a
finite partition of $\left[ 0,N\right] $ if $k_{j}\in 
\mathbb{N}
$ and $0=k_{0}<\cdots <k_{n}=N$. Denote the set of finite partitions of
interval $J$ as $D_{J}$.
\end{definition}

If two intervals only intersect on their boundary, then we abuse the notion
and label them as "\textit{disjoint}".

\begin{definition}
Interval $I$ is called a dyadic interval of level $n\in 
\mathbb{N}
$, if $I=\left[ k2^{n},\left( k+1\right) 2^{n}\right] $ for some $k\in 
\mathbb{N}
$. Integer$\ m$ is called a dyadic point of level $n\in 
\mathbb{N}
$, if $m=k2^{n}$ for some $k\in 
\mathbb{N}
$.
\end{definition}

\begin{notation}
\label{Notation of level of dyadic interval and point}For interval $J$,
denote the level of biggest dyadic interval in $J$ as $n\left( J\right) $,
i.e. $n\left( J\right) =\max \left\{ \text{level of dyadic interval }%
I|I\subseteq J\right\} $. Similarly, for $P\in 
\mathbb{N}
$, denote $N\left( P\right) :=\max \left\{ n|P=k2^{n}\text{ for }n\in 
\mathbb{N}
\text{, }k\in 
\mathbb{N}
\right\} $.
\end{notation}

Thus, $2^{n\left( J\right) }\leq \left\vert J\right\vert $, so $n\left(
J\right) \leq \log _{2}\left\vert J\right\vert $; $N\left( 0\right) =\infty $%
; $N\left( m\right) \geq 0$, $\forall m\in 
\mathbb{N}
$.

\begin{notation}
\label{Notation of DI_J}Suppose $J$ is a finite interval. Denote $B_{J}$ as
the set of dyadic intervals in $J$, i.e. $B_{J}:=\left\{ I|\text{ interval }I%
\text{ is dyadic, and }I\subseteq J\right\} $, and $B_{J}^{j}:=\left\{
I|I\in B_{J}\text{, }n\left( I\right) =j\right\} $.
\end{notation}

Then two properties of $B_{J}(B_{J}^{j})$.

\begin{itemize}
\item[$(i)$] Suppose $\left\{ I_{k}\right\} \in D_{J}$ (i.e.$\left\{
I_{k}\right\} $ is a finite partition of interval $J$), then $%
B_{I_{k_{1}}}\cap B_{I_{k_{2}}}=\varnothing $ when $k_{1}\neq k_{2}$, and%
\begin{equation}
\sqcup _{k}B_{I_{k}}\subseteq B_{J}.  \label{Property 1 of DI_J}
\end{equation}%
Similar result holds for $B_{J}^{j}$ for any level $j$:%
\begin{equation}
\sqcup _{k}B_{I_{k}}^{j}\subseteq B_{J}^{j}.  \label{Property of B_J^j}
\end{equation}
\end{itemize}

\begin{proof}
Only prove $\left( \ref{Property 1 of DI_J}\right) $; $\left( \ref{Property
of B_J^j}\right) $ is similar. $I_{k}\subseteq J$ so $B_{I_{k}}\subseteq
B_{J}$. $I_{k_{1}}$ and $I_{k_{2}}$ are disjoint when $k_{1}\neq k_{2}$, so $%
B_{I_{k_{1}}}\cap B_{I_{k_{2}}}=\varnothing $.
\end{proof}

\begin{itemize}
\item[$(ii)$] Let $X$ be the partial sum process of $\sum_{n=0}^{\infty
}c_{n}u_{n}$. Then for any interval $J$, (for interval $I$, denote $%
X_{\omega }\left( I\right) :=X_{\omega }\left( \sup I\right) -X_{\omega
}\left( \inf I\right) $, $\omega \in \Omega $)%
\begin{equation}
\sum_{I\in B_{J}}\int_{\Omega }\left\Vert X_{\omega }\left( I\right)
\right\Vert ^{2}\mu \left( d\omega \right) \leq 2\log _{2}\left( \left\vert
J\right\vert +1\right) \sum_{k,\left[ k-1,k\right] \subseteq J}\left\vert
c_{k}\right\vert ^{2}.  \label{Property II of DI_J}
\end{equation}
\end{itemize}

\begin{proof}
Each $\left[ k-1,k\right] \subseteq J$ can only be included in one dyadic
interval of level $j$, $0\leq j\leq n\left( J\right) $, so in $\cup
_{I}\left\{ I|I\in B_{J}\right\} ~$(the union of all dyadic intervals in $J$%
), $\left[ k-1,k\right] $ is counted at most $n\left( J\right) +1\leq \log
_{2}\left\vert J\right\vert +1\leq 2\log _{2}\left( \left\vert J\right\vert
+1\right) $ times. While for each interval $I$, 
\begin{equation*}
\int_{\Omega }\left\Vert X_{\omega }\left( I\right) \right\Vert ^{2}\mu
\left( d\omega \right) =\int_{\Omega }|\hspace{-0.01in}|\sum_{k,\left[ k-1,k%
\right] \subseteq I}c_{k}u_{k}\left( \omega \right) |\hspace{-0.01in}%
|^{2}\mu \left( d\omega \right) =\sum_{k,\left[ k-1,k\right] \subseteq
I}\left\vert c_{k}\right\vert ^{2},
\end{equation*}%
so sum over all dyadic intervals $I$ in $B_{J}$, 
\begin{eqnarray*}
\sum_{I\in B_{J}}\int_{\Omega }\left\Vert X_{\omega }\left( I\right)
\right\Vert ^{2}\mu \left( d\omega \right) &=&\sum_{I\in B_{J}}\sum_{k, 
\left[ k-1,k\right] \subseteq I}\left\vert c_{k}\right\vert ^{2} \\
&=&\sum_{k,\left[ k-1,k\right] \subseteq J}\#\left\{ I|\left[ k-1,k\right]
\subseteq I,I\in B_{J}\right\} \left\vert c_{k}\right\vert ^{2} \\
&\leq &2\log _{2}\left( \left\vert J\right\vert +1\right) \sum_{k,\left[
k-1,k\right] \subseteq J}\left\vert c_{k}\right\vert ^{2}.
\end{eqnarray*}
\end{proof}

The following two Lemmas give a method of decomposing an interval as union
of dyadic intervals: each time, we cut out biggest dyadic interval
available, and the number of dyadic sub-intervals is bounded above by
logarithm of the length of the interval. (The decomposition is in the same
spirit in Prop 4.1.1. in \cite{T. J. Lyons and Z. Qian}.)

\begin{lemma}
\label{Lemma preparation of dyadic partition}Suppose $J\ $is an interval
with one boundary point a level $n$ dyadic point $k2^{n}$, for some $k\geq 0$%
, $n\geq 1$, and $\left\vert J\right\vert <2^{n}$. Then, $J$ can be
decomposed as union of disjoint dyadic intervals, in such a way that the
level of dyadic intervals is strictly monotone with respect to their
position in $J$ (strictly increasing when $k2^{n}$ is the right boundary
point of $J$; strictly decreasing when $k2^{n}$ is the left boundary point
of $J$).
\end{lemma}

\begin{proof}
Set, for example, $J=\left[ k2^{n},a\right] $. Since $\left\vert
J\right\vert <2^{n}$, there exist $n>n_{1}>\cdots >n_{s}\geq 0$ such that $%
\left\vert J\right\vert =2^{n_{1}}+\cdots +2^{n_{s}}$. Then we decompose $J$
as%
\begin{eqnarray*}
J &=&\left[ k2^{n},k2^{n}+2^{n_{1}}\right] \cup \left[
k2^{n}+2^{n_{1}},k2^{n}+2^{n_{1}}+2^{n_{2}}\right] \\
&&\cup \cdots \cup \left[ k2^{n}+2^{n_{1}}+\cdots
+2^{n_{s-1}},k2^{n}+2^{n_{1}}+\cdots +2^{n_{s-1}}+2^{n_{s}}\right] \text{.}
\end{eqnarray*}
\end{proof}

\begin{lemma}
\label{Lemma of dyadic partition}Suppose $J$ is an interval, then there
exists a decomposition of $J$ as union of disjoint dyadic intervals, in a
way that there exists a point $P$ in the dyadic partition, such that $%
N\left( P\right) \geq n\left( J\right) +1$, and to the left and right side
of $P$, the level of dyadic intervals is strictly decreasing. As a result,
no more than two dyadic intervals of any given level are included, and the
number of dyadic intervals is bounded by $4\log _{2}\left( \left\vert
J\right\vert +1\right) $.
\end{lemma}

\begin{proof}
Denote $n_{0}:=n\left( J\right) $ (the level of biggest dyadic interval in $%
J $). Then there exists at least one dyadic interval of level $n_{0}$ in $J$%
, and there can be two adjacent ones, but there can not be more than two of
them. If there is one level $n_{0}$ interval, we select $P$ as the boundary
point of the level $n_{0}$ interval which satisfies $N\left( P\right) \geq
n\left( J\right) +1$. When there are two level $n_{0}$ intervals, we select $%
P$ as the point between these two level $n_{0}$ intervals (so $N\left(
P\right) \geq n\left( J\right) +1$). For the rest part (on the left and
right side of the level $n_{0}$ interval(s)), if they are not empty then
they are of the type in Lemma \ref{Lemma preparation of dyadic partition},
so can be decomposed accordingly. In this way, based on Lemma \ref{Lemma
preparation of dyadic partition}, the level of dyadic intervals is strictly
decreasing from $P$ to left and right. As a result, no more than two dyadic
intervals of any given level are included. Since $2^{n_{0}}\leq \left\vert
J\right\vert $, the number of dyadic intervals is bounded by $2n_{0}+2\leq
2\log _{2}\left\vert J\right\vert +2\leq 4\log _{2}\left( \left\vert
J\right\vert +1\right) $.
\end{proof}

\begin{lemma}
\label{Lemma relation between area of a path and its discretization}Suppose $%
\gamma :\left\{ 0,1,\dots ,N\right\} \rightarrow \mathcal{V}$ and $%
0=m_{0}<\cdots <m_{n}=N$. Define $\gamma ^{1}:\left\{ 0,1,\dots ,n\right\}
\rightarrow \mathcal{V}$ as $\gamma ^{1}\left( k\right) :=\gamma \left(
m_{k}\right) $, $k=0,1,\dots ,n$. Then%
\begin{equation}
A\left( \gamma \right) \left( 0,N\right) =A\left( \gamma ^{1}\right) \left(
0,n\right) +\sum_{k=0}^{n-1}A\left( \gamma \right) \left(
m_{k},m_{k+1}\right) \text{.}
\label{equality between area and its discretization}
\end{equation}
\end{lemma}

\begin{proof}
Based on the definition of area (Definition \ref{Definition of area} on p%
\pageref{Definition of area}), we have%
\begin{eqnarray*}
A\left( \gamma ^{1}\right) \left( 0,n\right) &=&\frac{1}{2}\sum_{0\leq
k<j\leq n-1}\left[ \gamma ^{1}\left( k+1\right) -\gamma ^{1}\left( k\right)
,\gamma ^{1}\left( j+1\right) -\gamma ^{1}\left( j\right) \right] \\
&=&\frac{1}{2}\sum_{j=1}^{n-1}\left[ \gamma \left( m_{j}\right) -\gamma
\left( m_{0}\right) ,\gamma \left( m_{j+1}\right) -\gamma \left(
m_{j}\right) \right] \text{.}
\end{eqnarray*}%
Then the equality $\left( \ref{equality between area and its discretization}%
\right) $ can be obtained by repeatedly applying multiplicativity of $\left(
\gamma ,A\left( \gamma \right) \right) $ (i.e. $\left( \ref{definition of
multiplicative functional}\right) $).
\end{proof}

\begin{lemma}
\label{Lemma: key inequality for path and area}Suppose $\gamma :%
\mathbb{N}
\rightarrow \mathcal{V}$ is a continuous path, and $\left\{ m_{n}\right\}
_{n=0}^{\infty }$ is a sequence of strictly increasing integers satisfying $%
\lim_{n\rightarrow \infty }m_{n}=+\infty $. Define $\gamma ^{1}:%
\mathbb{N}
\rightarrow \mathcal{V}$ as $\gamma ^{1}\left( n\right) :=\gamma \left(
m_{n}\right) $, $\forall n\in 
\mathbb{N}
$. Then 
\begin{equation}
\left\Vert \gamma \right\Vert _{2-var}^{2}\leq 3\left( \left\Vert \gamma
\right\Vert _{2-var,\left[ 0,m_{0}\right] }^{2}\right. {}+\sum_{n=0}^{\infty
}\left. \left\Vert \gamma \right\Vert _{2-var,\left[ m_{n},m_{n+1}\right]
}^{2}+\left\Vert \gamma ^{1}\right\Vert _{2-var}\right) ,
\label{Bound of 2-var by 2-var of two processes}
\end{equation}%
\begin{eqnarray}
\text{and \ \ }\left\Vert A\left( \gamma \right) \right\Vert _{1-var} &\leq
&\left\Vert \gamma \right\Vert _{2-var}^{2}+\left\Vert A\left( \gamma
\right) \right\Vert _{1-var,\left[ 0,m_{0}\right] }
\label{Bound for 1var of area} \\
&&+\sum_{n=0}^{\infty }\left\Vert A\left( \gamma \right) \right\Vert
_{1-var, \left[ m_{n},m_{n+1}\right] }+\left\Vert A\left( \gamma ^{1}\right)
\right\Vert _{1-var}.  \notag
\end{eqnarray}
\end{lemma}

\begin{proof}
For any finite interval $\left[ k_{1},k_{2}\right] $, if there exists $%
n_{1}\leq n_{2}$, s.t. $k_{1}<m_{n_{1}}\leq m_{n_{2}}<k_{2}$, then ($\gamma
\left( k_{1},k_{2}\right) :=\gamma \left( k_{2}\right) -\gamma \left(
k_{1}\right) $) 
\begin{equation*}
\left\Vert \gamma \left( k_{1},k_{2}\right) \right\Vert ^{2}\leq 3\left(
\left\Vert \gamma \left( k_{1},m_{n_{1}}\right) \right\Vert ^{2}+\left\Vert
\gamma \left( m_{n_{1}},m_{n_{2}}\right) \right\Vert ^{2}+\left\Vert \gamma
\left( m_{n_{2}},k_{2}\right) \right\Vert ^{2}\right) .
\end{equation*}%
Therefore, for any $N\geq 1$ and any fixed finite partition $\left\{ \left[
k_{j},k_{j+1}\right] \right\} _{j}\in D_{\left[ 0,m_{N}\right] }$, we take
the sum $\sum_{j}\left\Vert \gamma \left( k_{j},k_{j+1}\right) \right\Vert
^{2}$ and change $\left\Vert \gamma \left( k_{j},k_{j+1}\right) \right\Vert
^{2}$ into 
\begin{equation*}
2\left( \left\Vert \gamma \left( k_{j},m_{n}\right) \right\Vert
^{2}+\left\Vert \gamma \left( m_{n},k_{j+1}\right) \right\Vert ^{2}\right)
\end{equation*}%
whenever $\left( k_{j},k_{j+1}\right) $ contains one element $m_{n}$. We
change $\left\Vert \gamma \left( k_{j},k_{j+1}\right) \right\Vert ^{2}$ into 
$3\left( \left\Vert \gamma \left( k_{j},m_{n_{1}}\right) \right\Vert
^{2}+\left\Vert \gamma \left( m_{n_{1}},m_{n_{2}}\right) \right\Vert
^{2}+\left\Vert \gamma \left( m_{n_{2}},k_{j}\right) \right\Vert ^{2}\right) 
$ when $\left\{ m_{n}\right\} _{n=n_{1}}^{n_{2}}$ are all elements of $%
\left\{ m_{n}\right\} $ contained in $\left( k_{j},k_{j+1}\right) $.

Therefore, for any $N\geq 1$ and any fixed finite partition $\left\{ \left[
k_{j},k_{j+1}\right] \right\} _{j}\in D_{\left[ 0,m_{N}\right] }$, we have%
\begin{equation*}
\sum_{j}\left\Vert \gamma \left( k_{j},k_{j+1}\right) \right\Vert ^{2}\leq
3\left( \left\Vert \gamma \right\Vert _{2-var,\left[ 0,m_{0}\right]
}^{2}\right. {}+\sum_{n=0}^{N-1}\left. \left\Vert \gamma \right\Vert _{2-var,%
\left[ m_{n},m_{n+1}\right] }^{2}+\left\Vert \gamma ^{1}\right\Vert _{2-var,%
\left[ 0,N\right] }^{2}\right) \text{.}
\end{equation*}%
Take supremum over all possible finite partitions of $\left[ 0,m_{N}\right] $%
, and let $N$ tends to infinity, $\left( \ref{Bound of 2-var by 2-var of two
processes}\right) $ holds.

For $\left( \ref{Bound for 1var of area}\right) $, using multiplicativity of 
$\left( \gamma ,A\left( \gamma \right) \right) $ (i.e. $\left( \ref%
{definition of multiplicative functional}\right) $), we can get, if $%
k_{1}<m_{n_{1}}\leq m_{n_{2}}<k_{2}$, 
\begin{eqnarray*}
&&\left\Vert A\left( \gamma \right) \left( k_{1},k_{2}\right) \right\Vert \\
&\leq &\left\Vert A\left( \gamma \right) \left( k_{1},m_{n_{1}}\right)
\right\Vert +\left\Vert A\left( \gamma \right) \left(
m_{n_{1}},m_{n_{2}}\right) \right\Vert +\left\Vert A\left( \gamma \right)
\left( m_{n_{2}},k_{2}\right) \right\Vert \\
&&\qquad +\left\Vert \gamma \left( k_{1},m_{n_{1}}\right) \right\Vert
^{2}+\left\Vert \gamma \left( m_{n_{1}},m_{n_{2}}\right) \right\Vert
^{2}+\left\Vert \gamma \left( m_{n_{2}},k_{2}\right) \right\Vert ^{2} \\
&\leq &\left\Vert A\left( \gamma \right) \left( k_{1},m_{n_{1}}\right)
\right\Vert +\left\Vert A\left( \gamma \right) \left(
m_{n_{1}},m_{n_{2}}\right) \right\Vert +\left\Vert A\left( \gamma \right)
\left( m_{n_{2}},k_{2}\right) \right\Vert +\left\Vert \gamma \right\Vert
_{2-var,\left[ k_{1},k_{2}\right] }^{2}.
\end{eqnarray*}%
While based on Lemma \ref{Lemma relation between area of a path and its
discretization},%
\begin{equation*}
\left\Vert A\left( \gamma \right) \left( m_{n_{1}},m_{n_{2}}\right)
\right\Vert \leq \left\Vert A\left( \gamma ^{1}\right) \left(
n_{1},n_{2}\right) \right\Vert +\sum_{k=n_{1}}^{n_{2}-1}\left\Vert A\left(
\gamma \right) \left( m_{k},m_{k+1}\right) \right\Vert \text{.}
\end{equation*}%
The following reasoning is similar to that for $\left( \ref{Bound of 2-var
by 2-var of two processes}\right) $.
\end{proof}

The following Lemma works in the same spirit as the Lemma used in the proof
of Menshov-Rademacher theorem, but replace $\infty $-variation by $2$%
-variation.

\begin{lemma}
\label{Lemma estimation of 2-variation}Suppose $X$ is the partial sum
process of $\sum_{k=0}^{n}c_{n}u_{n}$, then, 
\begin{equation*}
\int_{\Omega }\left\Vert X_{\omega }\right\Vert _{2-var,\left[ 0,n\right]
}^{2}\mu \left( d\omega \right) \leq 8\left( \log _{2}\left( n+1\right)
\right) ^{2}\sum_{k=1}^{n}\left\vert c_{k}\right\vert ^{2}.
\end{equation*}
\end{lemma}

\begin{proof}
Suppose interval $J\subseteq \left[ 0,n\right] $. By Lemma \ref{Lemma of
dyadic partition}, decompose $J$ as union of disjoint dyadic intervals,
denote them as $I_{k}$, $1\leq k\leq l$, with $l\leq 4\log _{2}\left(
\left\vert J\right\vert +1\right) $. $B_{J}$ is the set of dyadic intervals
included in $J$ (Notation \ref{Notation of DI_J}). $I_{k}$ are disjoint, and
each $I_{k}$ is a member of $B_{J}$, so $\sum_{k=1}^{l}\left\Vert X_{\omega
}\left( I_{k}\right) \right\Vert ^{2}\leq \sum_{I\in B_{J}}\left\Vert
X_{\omega }\left( I\right) \right\Vert ^{2}$ for each $\omega \in \Omega $.
Then using Cauchy-Schwarz inequality, we get 
\begin{eqnarray}
\left\Vert X_{\omega }\left( J\right) \right\Vert ^{2} &=&|\hspace{-0.01in}%
|\sum_{k=1}^{l}X_{\omega }\left( I_{k}\right) |\hspace{-0.01in}|^{2}\leq
l\sum_{k=1}^{l}\left\Vert X_{\omega }\left( I_{k}\right) \right\Vert ^{2}
\label{inner tempt1} \\
&\leq &4\log _{2}\left( \left\vert J\right\vert +1\right)
\sum_{k=1}^{l}\left\Vert X_{\omega }\left( I_{k}\right) \right\Vert ^{2}\leq
4\log _{2}\left( n+1\right) \sum_{I\in B_{J}}\left\Vert X_{\omega }\left(
I\right) \right\Vert ^{2}.  \notag
\end{eqnarray}

Suppose $\left\{ J_{i}\right\} \in D_{\left[ 0,n\right] }$ (the set of
finite partitions of $\left[ 0,n\right] $). Use $\left( \ref{inner tempt1}%
\right) $ for each $J_{i}$, and $\sqcup _{i}B_{J_{i}}\subseteq B_{\left[ 0,n%
\right] }$ (according to $\left( \ref{Property 1 of DI_J}\right) $), we have 
\begin{equation}
\left\Vert X_{\omega }\right\Vert _{2-var,\left[ 0,n\right]
}^{2}=\sup_{\left\{ J_{i}\right\} \in D_{\left[ 0,n\right]
}}\sum_{i}\left\Vert X_{\omega }\left( J_{i}\right) \right\Vert ^{2}\leq
4\log _{2}\left( n+1\right) \sum_{I\in B_{\left[ 0,n\right] }}\left\Vert
X_{\omega }\left( I\right) \right\Vert ^{2}.  \notag
\end{equation}%
Integrate both sides, and use property at $\left( \ref{Property II of DI_J}%
\right) $, i.e. $\sum_{I\in B_{\left[ 0,n\right] }}\int_{\Omega }\left\Vert
X_{\omega }\left( I\right) \right\Vert ^{2}\mu \left( d\omega \right) \leq
2\log _{2}\left( n+1\right) \sum_{k=1}^{n}\left\vert c_{k}\right\vert ^{2}$,
we get 
\begin{equation*}
\int_{\Omega }\left\Vert X_{\omega }\right\Vert _{2-var,\left[ 0,n\right]
}^{2}\mu \left( d\omega \right) \leq 8\left( \log _{2}\left( n+1\right)
\right) ^{2}\sum_{k=1}^{n}\left\vert c_{k}\right\vert ^{2}.
\end{equation*}
\end{proof}

This inequality is interesting when taking into account that: (p255$\cite{B.
S. Kashin and A. A. Saakyan}$) there exists $c_{0}>0$ such that, for any $%
n\geq 1$ there exists an orthonormal sequence $\left\{ \varphi _{k}\right\}
_{k=1}^{n}$ on $\left( 0,1\right) $, s.t. the partial sum process $X^{n}$ of 
$\frac{1}{\sqrt{n}}\sum_{k=1}^{n}\varphi _{k}$ satisfies%
\begin{equation*}
P\left( \left\Vert X^{n}\right\Vert _{\infty -var}\geq c_{0}\log
_{2}n\right) \geq \frac{1}{4}.
\end{equation*}

The following result is proved in \cite{L. Allison}, we put it here for
completeness.

\begin{lemma}
\label{Lemma finite 2-variation Menshov-Rademacher theorem}The partial sum
process of $\sum_{n}c_{n}u_{n}$ (denoted as $X$) is of finite $2$-variation
a.e. for any orthonormal system $\left\{ u_{n}\right\} $ in $L^{2}$ and any
sequence of numbers $\left\{ c_{n}\right\} $ satisfying $\sum_{n}\left( \log
_{2}\left( n+1\right) \right) ^{2}\left\vert c_{n}\right\vert ^{2}<\infty $.
Moreover, $\left( \log _{2}\left( n+1\right) \right) ^{2}$ can not be
replaced by $o\left( \left( \log _{2}\left( n+1\right) \right) ^{2}\right) $
and%
\begin{equation}
\int_{\Omega }\left\Vert X_{\omega }\right\Vert _{2-var}^{2}\mu \left(
d\omega \right) \leq 36\sum_{n=0}^{\infty }\left( \log _{2}\left( n+1\right)
\right) ^{2}\left\vert c_{n}\right\vert ^{2}.
\label{Estimation of integration of 2-var of p.s.p}
\end{equation}
\end{lemma}

\begin{proof}
Since $\left\Vert X_{\omega }\right\Vert _{\infty -var}\leq \left\Vert
X_{\omega }\right\Vert _{2-var}$, $\forall \omega \in \Omega $, based on
Menshov-Rademacher Theorem, we only have to prove $\left( \ref{Estimation of
integration of 2-var of p.s.p}\right) $. Suppose $X$ takes value in $%
\mathcal{V}$. Define $X^{1}:%
\mathbb{N}
\rightarrow \mathcal{V}$ as $X^{1}\left( n\right) :=X\left( 2^{n}\right) $, $%
\forall n\in 
\mathbb{N}
$. Then according to $\left( \ref{Bound of 2-var by 2-var of two processes}%
\right) $ in Lemma \ref{Lemma: key inequality for path and area} (with $%
m_{n}=2^{n}$), 
\begin{eqnarray}
\int_{\Omega }\left\Vert X_{\omega }\right\Vert _{2-var}^{2}\mu \left(
d\omega \right) &\leq &3\int_{\Omega }\left\Vert X_{\omega }\right\Vert
_{2-var,\left[ 0,1\right] }^{2}\mu \left( d\omega \right)
\label{inner representation X} \\
&&+3\sum_{n=0}^{\infty }\int_{\Omega }\left\Vert X_{\omega }\right\Vert
_{2-var,\left[ 2^{n},2^{n+1}\right] }^{2}\mu \left( d\omega \right)
+3\int_{\Omega }\left\Vert X_{\omega }^{1}\right\Vert _{2-var}^{2}{}\mu
\left( d\omega \right) .  \notag
\end{eqnarray}%
While if denote $f$ as the limit function (according to Menshov-Rademacher
theorem, $f\left( \omega \right) =\lim_{n\rightarrow \infty }X_{\omega
}\left( n\right) $ exists a.e., set $f\left( \omega \right) =0$ elsewhere),
we have 
\begin{eqnarray*}
\int_{\Omega }\left\Vert X_{\omega }^{1}\right\Vert _{2-var}^{2}\mu \left(
d\omega \right) &=&\int_{\Omega }\sup_{\left\{ m_{k}\right\}
,m_{k}<m_{k+1}}\sum_{k}\left\Vert X_{\omega }\left( 2^{m_{k+1}}\right)
-X_{\omega }\left( 2^{m_{k}}\right) \right\Vert ^{2}\mu \left( d\omega
\right) \\
&\leq &2\int_{\Omega }\sum_{n=0}^{\infty }\left\Vert X_{\omega }\left(
2^{n}\right) -f\left( \omega \right) \right\Vert ^{2}\mu \left( d\omega
\right) =2\sum_{n=0}^{\infty }\sum_{k\geq 2^{n}+1}\left\vert
c_{k}\right\vert ^{2} \\
&\leq &4\sum_{n=2}^{\infty }\left( \log _{2}\left( n+1\right) \right)
^{2}\left\vert c_{n}\right\vert ^{2}.
\end{eqnarray*}%
Combined with Lemma \ref{Lemma estimation of 2-variation} for estimation of $%
\left\Vert X_{\omega }\right\Vert _{2-var,\left[ 2^{n},2^{n+1}\right] }^{2}$%
, $n\geq 0$, and $\left( \ref{inner representation X}\right) $, proof
finishes.
\end{proof}

We will use Lemma \ref{Lemma finite 2-variation Menshov-Rademacher theorem}
in the proof of Theorem \ref{Theorem partial sum process of general
orthogonal expansion as geometric rough process}.

\bigskip

\begin{proof}[Proof of Theorem \protect\ref{Theorem partial sum process of
general orthogonal expansion as geometric rough process}]
Denote the partial sum process of $\sum_{n=0}^{\infty }c_{n}u_{n}$ as $X$,
and denote $A:=A\left( X\right) $ as the area process of $X$. Since $%
\left\Vert X_{\omega }\right\Vert _{\infty -var}\leq \left\Vert X_{\omega
}\right\Vert _{2-var}$, $\forall \omega \in \Omega $, based on
Menshov-Rademacher Theorem (on p\pageref{Menshov-Rademacher Theorem}), we
only need to prove $\int_{\Omega }(\left\Vert X_{\omega }\right\Vert
_{2-var}^{2}+\left\Vert A_{\omega }\right\Vert _{1-var})\mu \left( d\omega
\right) \leq 121\sum_{n=0}^{\infty }\left( \log _{2}\left( n+1\right)
\right) ^{2}\left\vert c_{n}\right\vert ^{2}$. While $\int_{\Omega
}\left\Vert X_{\omega }\right\Vert _{2-var}^{2}\mu \left( d\omega \right) $
is done in Lemma \ref{Lemma finite 2-variation Menshov-Rademacher theorem},
so we concentrate on the area. Define $X^{1}:%
\mathbb{N}
\rightarrow \mathcal{V}$ as $X^{1}\left( n\right) :=X\left( 2^{n}\right) $, $%
\forall n\in 
\mathbb{N}
$, and denote $A^{1}:=A\left( X^{1}\right) $. Then use $\left( \ref{Bound
for 1var of area}\right) $ in Lemma \ref{Lemma: key inequality for path and
area}, we have 
\begin{eqnarray}
\int_{\Omega }\left\Vert A_{\omega }\right\Vert _{1-var}\mu \left( d\omega
\right) &\leq &\int_{\Omega }\left\Vert X_{\omega }\right\Vert
_{2-var}^{2}\mu \left( d\omega \right) +\int_{\Omega }\left\Vert A_{\omega
}\right\Vert _{1-var,\left[ 0,1\right] }\mu \left( d\omega \right)
\label{inner bound for 1var of area} \\
&&+\sum_{n=0}^{\infty }\int_{\Omega }\left\Vert A_{\omega }\right\Vert
_{1-var,\left[ 2^{n},2^{n+1}\right] }\mu \left( d\omega \right)
+\int_{\Omega }\left\Vert A_{\omega }^{1}\right\Vert _{1-var}\mu \left(
d\omega \right) .  \notag
\end{eqnarray}%
Based on the definition of area process (Definition \ref{Definition of area}
on p\pageref{Definition of area}), $\left\Vert A_{\omega }\right\Vert
_{1-var,\left[ 0,1\right] }=\left\Vert A_{\omega }\right\Vert _{1-var,\left[
1,2\right] }=0$, $\forall \omega \in \Omega $. Thus, we are done if we can
prove 
\begin{equation}
\int_{\Omega }\left\Vert A_{\omega }\right\Vert _{1-var,\left[ 2^{n},2^{n+1}%
\right] }\mu \left( d\omega \right) \leq 6\sum_{k=2^{n}+1}^{2^{n+1}}\left(
\log _{2}\left( k+1\right) \right) ^{2}\left\vert c_{k}\right\vert ^{2}\text{%
, }\forall n\geq 1  \label{inner, local regularity}
\end{equation}%
and 
\begin{equation}
\int_{\Omega }\left\Vert A_{\omega }^{1}\right\Vert _{1-var}\mu \left(
d\omega \right) \leq 43\sum_{n=0}^{\infty }\left( \log _{2}\left( n+1\right)
\right) ^{2}\left\vert c_{n}\right\vert ^{2}.
\label{inner long time behavior}
\end{equation}%
($36+6+43=85$, $85+36=121$.)

In the following, we do analysis for fixed $\omega \in \Omega $.

Using multiplicativity of $\left( X_{\omega },A_{\omega }\right) $ (identity 
$\left( \ref{definition of multiplicative functional}\right) $ on p$\pageref%
{definition of multiplicative functional}$), for any finite interval $J$ and
any disjoint decomposition $\left\{ J_{1},J_{2}\right\} \in D_{J}$, we have%
\begin{equation*}
\left\Vert A_{\omega }\left( J\right) \right\Vert \leq \left\Vert A_{\omega
}\left( J_{1}\right) \right\Vert +\left\Vert A_{\omega }\left( J_{2}\right)
\right\Vert +\left\Vert X_{\omega }\left( J_{1}\right) \right\Vert
\left\Vert X_{\omega }\left( J_{2}\right) \right\Vert .
\end{equation*}

Therefore, for $A_{\omega }$ on dyadic interval $I=\left[ m2^{n},\left(
m+1\right) 2^{n}\right] $, by repeatedly bisecting $I$ and using $A\left(
k,k+1\right) =0$, $\forall k\in 
\mathbb{N}
$, we get ($B_{I}$ is the set of dyadic intervals included in $I$, Notation $%
\ref{Notation of DI_J}$),%
\begin{eqnarray}
\left\Vert A_{\omega }\left( I\right) \right\Vert &=&\left\Vert A_{\omega
}\left( m2^{n},\left( m+1\right) 2^{n}\right) \right\Vert
\label{inner estimation of area on dyadic interal} \\
&\leq &\sum_{j=0}^{n-1}\sum_{k=0}^{2^{n-j}-1}\left\Vert X_{\omega }\left( 
\left[ m2^{n}+k2^{j},m2^{n}+\left( k+1\right) 2^{j}\right] \right)
\right\Vert ^{2}  \notag \\
&\leq &\sum_{I^{\prime }\in B_{I}\backslash \left\{ I\right\} }\left\Vert
X_{\omega }\left( I^{\prime }\right) \right\Vert ^{2}.  \notag
\end{eqnarray}%
This estimation of $A_{\omega }$ on dyadic intervals will be used repeatedly.

For interval $J$ which is not dyadic, decompose it as union of dyadic
intervals $\left\{ I_{k}\right\} _{k=1}^{l}$ by Lemma \ref{Lemma of dyadic
partition} with $l\leq 4\log _{2}\left( \left\vert J\right\vert +1\right) $.
We estimate $A_{\omega }\left( J\right) $ by successively removing dyadic
partition points from $J$. Suppose $\left\{ I_{k}\right\} $ are numbered
that $k$ is increasing from left to right of $J$, then the accumulated error
incurred to $\left\Vert A_{\omega }\left( J\right) \right\Vert $ from
removing point between $I_{k}$ and $\cup _{j=k+1}^{l}I_{j}$, $1\leq k\leq
l-1 $, is bounded by 
\begin{gather}
\sum_{k=1}^{l-1}\left\Vert X_{\omega }\left( I_{k}\right) \right\Vert \cdot
\left\Vert X_{\omega }\left( \cup _{j=k+1}^{l}I_{j}\right) \right\Vert \leq
\sum_{k=1}^{l-1}\sum_{j=k+1}^{l}\left\Vert X_{\omega }\left( I_{k}\right)
\right\Vert \cdot \left\Vert X_{\omega }\left( I_{j}\right) \right\Vert
\label{inner error from removing dyadic partition points} \\
\leq \frac{1}{2}\left( \sum_{k=1}^{l-1}\sum_{j=k+1}^{l}\left( \left\Vert
X_{\omega }\left( I_{k}\right) \right\Vert ^{2}+\left\Vert X_{\omega }\left(
I_{j}\right) \right\Vert ^{2}\right) \right)  \notag \\
=\frac{1}{2}\sum_{k=1}^{l-1}\left( l-k\right) \left\Vert X_{\omega }\left(
I_{k}\right) \right\Vert ^{2}+\frac{1}{2}\sum_{j=2}^{l}\left( j-1\right)
\left\Vert X_{\omega }\left( I_{j}\right) \right\Vert ^{2}  \notag \\
\leq \frac{1}{2}l\sum_{k=1}^{l}\left\Vert X_{\omega }\left( I_{k}\right)
\right\Vert ^{2}\leq 2\log _{2}\left( \left\vert J\right\vert +1\right)
\sum_{k=1}^{l}\left\Vert X_{\omega }\left( I_{k}\right) \right\Vert ^{2}. 
\notag
\end{gather}%
After removing all dyadic partition points from $J$, we are left with area
on $I_{k}$, $1\leq k\leq l$, so 
\begin{equation*}
\left\Vert A_{\omega }\left( J\right) \right\Vert \leq
\sum_{k=1}^{l}\left\Vert A_{\omega }\left( I_{k}\right) \right\Vert +2\log
_{2}\left( \left\vert J\right\vert +1\right) \sum_{k=1}^{l}\left\Vert
X_{\omega }\left( I_{k}\right) \right\Vert ^{2}.
\end{equation*}%
Then apply $\left( \ref{inner estimation of area on dyadic interal}\right) $
to each $I_{k}$, and use $\sqcup _{k=1}^{l}\left\{ I_{k}\right\} \subseteq
\sqcup _{k=1}^{l}B_{I_{k}}\subseteq B_{J}$ (since $I_{k}$ are dyadic and $%
\left\{ I_{k}\right\} _{k=1}^{l}$ is a finite partition of $J$, use $\left( %
\ref{Property 1 of DI_J}\right) $), 
\begin{equation*}
\sum_{k=1}^{l}\left\Vert A_{\omega }\left( I_{k}\right) \right\Vert \leq
\sum_{k=1}^{l}\sum_{I\in B_{I_{k}}}\left\Vert X_{\omega }\left( I\right)
\right\Vert ^{2}\leq \sum_{I\in B_{J}}\left\Vert X_{\omega }\left( I\right)
\right\Vert ^{2}.
\end{equation*}%
\begin{eqnarray}
\text{Thus, }\left\Vert A_{\omega }\left( J\right) \right\Vert &\leq
&\sum_{I\in B_{J}}\left\Vert X_{\omega }\left( I\right) \right\Vert
^{2}+2\log _{2}\left( \left\vert J\right\vert +1\right) \sum_{I\in
B_{J}}\left\Vert X_{\omega }\left( I\right) \right\Vert ^{2}
\label{inner tempt2} \\
&\leq &3\log _{2}\left( \left\vert J\right\vert +1\right) \sum_{I\in
B_{J}}\left\Vert X_{\omega }\left( I\right) \right\Vert ^{2}.  \notag
\end{eqnarray}

Therefore, suppose $\left\{ J_{i}\right\} _{i}\in D_{\left[ 2^{n},2^{n+1}%
\right] }$, $n\geq 1$, apply $\left( \ref{inner tempt2}\right) $ for each $%
J_{i}$, and use $\sqcup _{i}B_{J_{i}}\subseteq B_{\left[ 2^{l},2^{l+1}\right]
}$, we get 
\begin{equation*}
\sum_{i}\left\Vert A_{\omega }\left( J_{i}\right) \right\Vert \leq
\sum_{i}3\log _{2}\left( \left\vert J_{i}\right\vert +1\right) \sum_{I\in
B_{J_{i}}}\left\Vert X_{\omega }\left( I\right) \right\Vert ^{2}\leq 3\log
_{2}\left( 2^{n}+1\right) \sum_{I\in B_{[2^{n},2^{n+1}]}}\left\Vert
X_{\omega }\left( I\right) \right\Vert ^{2}.
\end{equation*}%
Taking supremum over all finite partitions,%
\begin{equation*}
\left\Vert A_{\omega }\right\Vert _{1-var,\left[ 2^{n},2^{n+1}\right]
}=\sup_{\left\{ J_{i}\right\} \in D_{[2^{n},2^{n+1}]}}\sum_{i}\left\Vert
A_{\omega }\left( J_{i}\right) \right\Vert \leq 3\log _{2}\left(
2^{n}+1\right) \sum_{I\in B_{[2^{n},2^{n+1}]}}\left\Vert X_{\omega }\left(
I\right) \right\Vert ^{2}.
\end{equation*}%
Integrate both sides, use $\left( \ref{Property II of DI_J}\right) $, i.e. 
\begin{equation*}
\sum_{I\in B_{[2^{n},2^{n+1}]}}\int_{\Omega }\left\Vert X_{\omega }\left(
I\right) \right\Vert ^{2}\mu \left( d\omega \right) \leq 2\log _{2}\left(
2^{n}+1\right) \sum_{k=2^{n}+1}^{2^{n+1}}\left\vert c_{k}\right\vert ^{2},
\end{equation*}%
and $\log _{2}\left( 2^{n}+1\right) \leq \log _{2}\left( k+1\right) $ when $%
k\in \left[ 2^{n},2^{n+1}\right] $, we get%
\begin{equation}
\int_{\Omega }\left\Vert A_{\omega }\right\Vert _{1-var,\left[ 2^{n},2^{n+1}%
\right] }\mu \left( d\omega \right) \leq 6\sum_{k=2^{n}+1}^{2^{n+1}}\left(
\log _{2}\left( k+1\right) \right) ^{2}\left\vert c_{k}\right\vert ^{2}.
\label{inner estimation for the variation of area on [2l,2l+1]}
\end{equation}

Then, what left is the estimation of the long-time behavior, i.e. $\left( %
\ref{inner long time behavior}\right) $ about $A^{1}:=A\left( X^{1}\right) $:%
\begin{equation*}
\int_{\Omega }\left\Vert A_{\omega }^{1}\right\Vert _{1-var}\mu \left(
d\omega \right) \leq 43\sum_{n=0}^{\infty }\left( \log _{2}\left( n+1\right)
\right) ^{2}\left\vert c_{n}\right\vert ^{2}.
\end{equation*}%
As\ it is defined, $X^{1}\left( n\right) =X\left( 2^{n}\right) $, $\forall
n\in 
\mathbb{N}
$. Thus, if denote 
\begin{equation}
b_{n}:=\sqrt{\sum_{k=2^{n}+1}^{2^{n+1}}\left\vert c_{k}\right\vert ^{2}}%
\text{ and }v_{n}\left( \omega \right) :=\sum_{k=2^{n}+1}^{2^{n+1}}\frac{%
c_{k}u_{k}\left( \omega \right) }{b_{n}},
\label{inner definition of vn and bn}
\end{equation}%
then $\left\{ v_{n}\right\} $ is an orthonormal system in $L^{2}$, and $%
X^{1} $ is the partial sum process of $\sum_{n=0}^{\infty }b_{n}v_{n}$.

To estimate $A^{1}$, since we already have an estimation of the local
behavior (i.e. $\left( \ref{inner estimation for the variation of area on
[2l,2l+1]}\right) $), we only need to work on its long term behavior. Denote 
$X^{2}:%
\mathbb{N}
\rightarrow \mathcal{V}$ by assigning $X^{2}\left( n\right) :=X^{1}\left(
2^{n}\right) $, $\forall n\in 
\mathbb{N}
$, and denote $A^{2}:=A\left( X^{2}\right) $. Then if denote%
\begin{equation}
a_{n}:=\sqrt{\sum_{k=2^{n}+1}^{2^{n+1}}\left\vert b_{n}\right\vert ^{2}}\ 
\text{and }r_{n}\left( \omega \right) :=\sum_{k=2^{n}+1}^{2^{n+1}}\frac{%
b_{k}v_{k}\left( \omega \right) }{a_{n}}\text{,}
\label{inner definition of an and rn}
\end{equation}%
then $X^{2}$ is the partial sum process of $\sum_{n=0}^{\infty }a_{n}r_{n}$.

Based on $\left( \ref{Bound for 1var of area}\right) $ on p\pageref{Bound
for 1var of area}, we have, (using $\left\Vert A^{1}\right\Vert _{1-var,%
\left[ 0,1\right] }=\left\Vert A^{1}\right\Vert _{1-var,\left[ 1,2\right]
}=0 $)%
\begin{eqnarray}
\int_{\Omega }\left\Vert A_{\omega }^{1}\right\Vert _{1-var}\mu \left(
d\omega \right) &\leq &\int_{\Omega }\left\Vert X_{\omega }^{1}\right\Vert
_{2-var}^{2}\mu \left( d\omega \right)  \label{inner estimation of A1(1)} \\
&&+\sum_{n=1}^{\infty }\int_{\Omega }\left\Vert A_{\omega }^{1}\right\Vert
_{1-var,\left[ 2^{n},2^{n+1}\right] }\mu \left( d\omega \right)
+\int_{\Omega }\left\Vert A_{\omega }^{2}\right\Vert _{1-var}\mu \left(
d\omega \right) .  \notag
\end{eqnarray}%
For $\int_{\Omega }\left\Vert X_{\omega }^{1}\right\Vert _{2-var}^{2}\mu
\left( d\omega \right) $, according to Lemma \ref{Lemma finite 2-variation
Menshov-Rademacher theorem} and the definition of $b_{n}$ at $\left( \ref%
{inner definition of vn and bn}\right) $, we have%
\begin{equation}
\int_{\Omega }\left\Vert X_{\omega }^{1}\right\Vert _{2-var}^{2}\mu \left(
d\omega \right) \leq 36\sum_{n=0}^{\infty }\left( \log _{2}\left( n+1\right)
\right) ^{2}\left\vert b_{n}\right\vert ^{2}\leq 36\sum_{n=0}^{\infty
}\left( \log _{2}\left( n+1\right) \right) ^{2}\left\vert c_{n}\right\vert
^{2}\text{.}  \label{inner estimation of A1(2)}
\end{equation}%
Similarly, for the accumulative effect of local behavior of $A^{1}$, based
on $\left( \ref{inner estimation for the variation of area on [2l,2l+1]}%
\right) $, we have%
\begin{equation}
\sum_{n=1}^{\infty }\int_{\Omega }\left\Vert A_{\omega }^{1}\right\Vert
_{1-var,\left[ 2^{n},2^{n+1}\right] }\mu \left( d\omega \right) \leq
6\sum_{n=0}^{\infty }\left( \log _{2}\left( n+1\right) \right)
^{2}\left\vert c_{n}\right\vert ^{2}\text{.}
\label{inner estimation of A1(3)}
\end{equation}%
Thus, if we can prove 
\begin{equation}
\int_{\Omega }\left\Vert A_{\omega }^{2}\right\Vert _{1-var}\mu \left(
d\omega \right) \leq \sum_{n=0}^{\infty }\left( \log _{2}\left( n+1\right)
\right) ^{2}\left\vert c_{n}\right\vert ^{2}\text{,}
\label{inner estimation of A1(4)}
\end{equation}%
then combined with $\left( \ref{inner estimation of A1(1)}\right) $, $\left( %
\ref{inner estimation of A1(2)}\right) $ and $\left( \ref{inner estimation
of A1(3)}\right) $, we can prove $\left( \ref{inner long time behavior}%
\right) $. ($36+6+1=43$.)

To prove $\left( \ref{inner estimation of A1(4)}\right) $, since%
\begin{equation*}
\sum_{n=0}^{\infty }4^{n}\left\vert a_{n}\right\vert ^{2}\leq
\sum_{n=0}^{\infty }n^{2}\left\vert b_{n}\right\vert ^{2}\leq
\sum_{n=0}^{\infty }\left( \log _{2}\left( n+1\right) \right) ^{2}\left\vert
c_{n}\right\vert ^{2}\text{,}
\end{equation*}%
we are done if we can prove%
\begin{equation}
\int_{\Omega }\left\Vert A_{\omega }^{2}\right\Vert _{1-var}\mu \left(
d\omega \right) \leq \sum_{n=0}^{\infty }4^{n}\left\vert a_{n}\right\vert
^{2}\text{.}  \label{inner estimation of A1(5)}
\end{equation}%
Actually, (with $\left\{ a_{n}\right\} $ and $\left\{ r_{n}\right\} $
defined at $\left( \ref{inner definition of an and rn}\right) $)%
\begin{eqnarray*}
\left\Vert A_{\omega }^{2}\right\Vert _{1-var} &\leq &\sum_{1\leq i<j<\infty
}\left\Vert a_{i}r_{i}\left( \omega \right) \right\Vert \left\Vert
a_{j}r_{j}\left( \omega \right) \right\Vert \\
&\leq &\sum_{1\leq i<j<\infty }2^{-\left( i+j\right) }\left( \left\Vert
2^{i}a_{i}r_{i}\left( \omega \right) \right\Vert ^{2}+\left\Vert
2^{j}a_{j}r_{j}\left( \omega \right) \right\Vert ^{2}\right) \\
&\leq &\sum_{i\geq 1}2^{-2i}\left\Vert 2^{i}a_{i}r_{i}\left( \omega \right)
\right\Vert ^{2}+\sum_{j\geq 1}2^{-j}\left\Vert 2^{j}a_{j}r_{j}\left( \omega
\right) \right\Vert ^{2} \\
&\leq &\sum_{n=0}^{\infty }4^{n}\left\vert a_{n}\right\vert ^{2}\left\Vert
r_{n}\left( \omega \right) \right\Vert ^{2}\text{.}
\end{eqnarray*}%
Thus, $\left( \ref{inner estimation of A1(5)}\right) $ and $\left( \ref%
{inner estimation of A1(4)}\right) $ holds, so $\left( \ref{inner long time
behavior}\right) $\ holds.

As a result, we have%
\begin{equation}
\int_{\Omega }\left\Vert A_{\omega }\right\Vert _{1-var}\mu \left( d\omega
\right) \leq 85\sum_{n=0}^{\infty }\left( \log _{2}\left( n+1\right) \right)
^{2}\left\vert c_{n}\right\vert ^{2},
\label{estimation of integration of 1var of area}
\end{equation}%
and%
\begin{equation*}
\int_{\Omega }\left\Vert \mathbf{X}_{\omega }\right\Vert _{G^{\left(
2\right) }}\mu \left( d\omega \right) \leq 121\sum_{n=0}^{\infty }\left(
\log _{2}\left( n+1\right) \right) ^{2}\left\vert c_{n}\right\vert ^{2}\text{%
.}
\end{equation*}
\end{proof}

The following decomposition is used in Theorem 16 \cite{L. Allison} to prove
the first part of our Theorem \ref{Theorem partial sum process of complete
orthogonal expansion as geometric rough process} (finiteness of $2$%
-variation of partial sum process of orthonormal systems satisfying Hardy
property, i.e. $\left( \ref{Hardy property}\right) $).

\begin{lemma}
\label{Lemma bisection partition}Every non-dyadic interval $J$ can be
decomposed as disjoint union of two intervals $J=J^{1}\cup J^{2}$, such that
there exist two disjoint dyadic intervals $I_{1}$ and $I_{2}$, satisfying $%
J^{i}\subseteq I^{i}$ and $\left\vert J^{i}\right\vert >\frac{1}{2}%
\left\vert I^{i}\right\vert $, $i=1,2$.
\end{lemma}

\begin{proof}
Based on Lemma \ref{Lemma of dyadic partition}, there exists a dyadic point $%
P$ in $J$ satisfying $N\left( P\right) \geq n\left( J\right) +1$. (In
particular, $P=0$ for $J=\left[ 0,n\right] $, as $N\left( 0\right) =\infty $%
.) If $P$ divides $J$ into two non-empty intervals (denoted as $J^{1}$ and $%
J^{2}$), then since the level of dyadic intervals is strictly decreasing
from $P$ to left/right and $N\left( P\right) \geq n\left( J\right) +1$, $%
J^{1}$ and $J^{2}$ satisfy the requirement. If $P$ is a boundary point of $J$%
, then the level of dyadic intervals in $J$ is already monotone. In that
case, we let $J^{1}$ be the biggest dyadic interval in $J$ and $%
J^{2}=J\backslash J^{1}$.
\end{proof}

\begin{remark}
\label{Remark level of dyadic interval is decreasing}As we selected, the
point in $J$ dividing $J^{1}$and $J^{2}$ is one of the boundary points of
biggest dyadic sub-interval(s) in $J$, and the level of dyadic intervals is
strictly decreasing to left and right side of this point.
\end{remark}

\begin{lemma}
\label{Lemma repeatedly bisecting non-dyadic interval got the same dyadic
partition}Suppose $J$ is a finite non-dyadic interval. If we bisect $%
J=J^{1}\cup J^{2}$ according to Lemma \ref{Lemma bisection partition}, and
continue to bisect $J^{1}$ and $J^{2}$ if they are non-dyadic, so on and so
forth, until all intervals left are dyadic. Then the dyadic intervals left
constitute the dyadic partition of $J$ by Lemma \ref{Lemma of dyadic
partition}.
\end{lemma}

\begin{proof}
Suppose the dyadic partition of $J$ by Lemma \ref{Lemma of dyadic partition}
is $\left\{ I_{k}\right\} _{k=1}^{n}$, where $\left\{ I_{k}\right\} $ are
numbered that $k\ $is increasing from left to right of $J$. Denote $P$ as
the point bisecting $J=J^{1}\cup J^{2}$ by Lemma \ref{Lemma bisection
partition}. Based Remark \ref{Remark level of dyadic interval is decreasing}%
, $P$ is one of the boundary points of some $I_{k}$, $1\leq k\leq n$, and
the level of dyadic intervals is strictly decreasing to the left and right
side of $P$. Since $\left\{ I_{k}\right\} _{k=1}^{n}$ is a finite partition
of $J$, $J^{1}$ and $J^{2}$ are union of $I_{k}$s: there exists $m$, $1\leq
m\leq n-1$, such that $J^{1}=\cup _{k=1}^{m}I_{k}$ and $J^{2}=\cup
_{k=m+1}^{n}I_{k}$. We continue to bisect $J^{1}$ and $J^{2}$ if they are
non-dyadic. Take $J^{1}=\cup _{k=1}^{m}I_{k}$ for example. Since $I_{k}$, $%
k=1,2,\dots ,m$ are strictly increasing in their level, according to Lemma %
\ref{Lemma bisection partition}, bisecting $J^{1}$ is to cut $I_{m}$ out
(the biggest dyadic subinterval). While $J^{1}\backslash I_{m}=\cup
_{k=1}^{m-1}I_{k}$ is still composed of strictly increasing dyadic
subintervals, so bisecting $J^{1}\backslash I_{m}$ is to cut $I_{m-1}$ out,
so on and so forth. In this way, bisecting $J^{1}$ down to dyadic intervals,
one gets back $\left\{ I_{k}\right\} _{k=1}^{m}$. Similar reasoning applies
to $J^{2}$.
\end{proof}

\bigskip

Before proceeding to the proof of Theorem \ref{Theorem partial sum process
of complete orthogonal expansion as geometric rough process}, we define $%
\widetilde{B}_{J}$ for finite interval $J$ as the set of dyadic intervals
which contain "part" of $J$.

\begin{notation}
\medskip Suppose $J$ is a finite interval, denote 
\begin{equation}
\widetilde{B}_{J}:=\left\{ I\bigg|I\text{ is dyadic, }\left\vert I\cap
J\right\vert >\frac{1}{2}\left\vert I\right\vert \right\} .
\label{Definition of B_J_tile}
\end{equation}
\end{notation}

\noindent Four properties of $\widetilde{B}_{J}$:

\begin{itemize}
\item[$(i)$] $B_{J}\subseteq \widetilde{B}_{J}$.

\item[$(ii)$] \label{Second property of B_J_tile}When $J$ is dyadic, $%
\widetilde{B}_{J}=B_{J}$.
\end{itemize}

\begin{proof}
For two dyadic intervals, either one is wholly included in another, or they
are disjoint, bar boundary points. Thus, suppose $J$ and $I$ are dyadic
intervals and $\left\vert I\cap J\right\vert >0$, then either $I\subseteq J$%
, or $J\subset I$. If $I\subseteq J$, then $I\in B_{J}\subseteq \widetilde{B}%
_{J}$. If $J\subset I$, and $I\in \widetilde{B}_{J}$, then $\left\vert
J\right\vert <\left\vert I\right\vert <2\left\vert I\cap J\right\vert
=2\left\vert J\right\vert $, which is not possible because $I$ and $J$ are
dyadic. Therefore, when $J$ is dyadic, $\widetilde{B}_{J}$ is the set of
dyadic intervals included in $J$, thus coincides with $B_{J}$.
\end{proof}

\begin{itemize}
\item[$(iii)$] \label{Third property of B_J_tile}If $J^{\prime }\subseteq J$%
, then $\widetilde{B}_{J^{\prime }}\subseteq \widetilde{B}_{J}$.
\end{itemize}

\begin{proof}
Suppose $I\in \widetilde{B}_{J^{\prime }}$, then $\left\vert I\cap
J\right\vert \geq \left\vert I\cap J^{\prime }\right\vert >\frac{1}{2}%
\left\vert I\right\vert $, so $I\in \widetilde{B}_{J}$.
\end{proof}

\begin{itemize}
\item[$(iv)$] \label{Fourth property of B_J_Tile}Suppose $\left\{
I_{k}\right\} $ is a finite partition of $J$, then $\sqcup _{k}\widetilde{B}%
_{I_{k}}\subseteq \widetilde{B}_{J}$.
\end{itemize}

\begin{proof}
$\widetilde{B}_{I_{k}}\subseteq \widetilde{B}_{J}$ is from $\left(
iii\right) $. If $I\in \widetilde{B}_{I_{k_{1}}}\cap \widetilde{B}%
_{I_{k_{2}}}$, $k_{1}\neq k_{2}$, then%
\begin{eqnarray*}
\left\vert I_{k_{1}}\cap I_{k_{2}}\right\vert &\geq &\left\vert \left( I\cap
I_{k_{1}}\right) \cap \left( I\cap I_{k_{2}}\right) \right\vert \\
&=&\left\vert I\cap I_{k_{1}}\right\vert +\left\vert I\cap
I_{k_{2}}\right\vert -\left\vert \left( I\cap I_{k_{1}}\right) \cup \left(
I\cap I_{k_{2}}\right) \right\vert \\
&>&\frac{1}{2}\left\vert I\right\vert +\frac{1}{2}\left\vert I\right\vert
-\left\vert I\right\vert =0,
\end{eqnarray*}%
contradictory with that $I_{k}$ are disjoint since $\left\{ I_{k}\right\} $
is a finite partition of $J$.
\end{proof}

\bigskip

\begin{proof}[Proof of Theorem \protect\ref{Theorem partial sum process of
complete orthogonal expansion as geometric rough process}]
Denote the partial sum process of $\sum_{n}c_{n}u_{n}$ as $X$, and $%
A:=A\left( X\right) $. Define process $X^{1}$ by assigning $X^{1}\left(
n\right) :=X\left( 2^{n}\right) $, $\forall n\in 
\mathbb{N}
$, and denote $A^{1}:=A\left( X^{1}\right) $. If let 
\begin{equation*}
v_{n}\left( \omega \right) =\sum_{k=2^{n}+1}^{2^{n+1}}\frac{c_{k}u_{k}\left(
\omega \right) }{\sqrt{\sum_{k=2^{n}+1}^{2^{n+1}}\left\vert c_{k}\right\vert
^{2}}}\text{ and }b_{n}=\sqrt{\sum_{k=2^{n}+1}^{2^{n+1}}\left\vert
c_{k}\right\vert ^{2}}\text{,}
\end{equation*}%
then $X^{1}$ is the partial sum process of $\sum_{n=0}^{\infty }b_{n}v_{n}$.
According to Theorem \ref{Theorem partial sum process of general orthogonal
expansion as geometric rough process}, $X^{1}$ is a geometric $2$-rough
process when $\sum_{n\geq 0}\left( \log _{2}(n+1)\right) ^{2}\left\vert
b_{n}\right\vert ^{2}<\infty $. On the other hand, (use $\left( \log
_{2}\left( n+1\right) \right) ^{2}\leq 2n$, $\forall n\in 
\mathbb{N}
$)%
\begin{equation*}
\sum_{n\geq 0}\left( \log _{2}\left( n+1\right) \right) ^{2}\left\vert
b_{n}\right\vert ^{2}\leq 2\sum_{n\geq 1}n\left\vert b_{n}\right\vert
^{2}\leq 2\sum_{n\geq 0}\log _{2}(n+1)\left\vert c_{n}\right\vert ^{2}\text{.%
}
\end{equation*}%
Thus when $\sum_{n}\log _{2}\left( n+1\right) \left\vert c_{n}\right\vert
^{2}<\infty $, $X^{1}$ is a geometric $2$-rough process, and (according to
Lemma \ref{Lemma finite 2-variation Menshov-Rademacher theorem} on p\pageref%
{Lemma finite 2-variation Menshov-Rademacher theorem} and $\left( \ref%
{estimation of integration of 1var of area}\right) $ on p\pageref{estimation
of integration of 1var of area}) 
\begin{eqnarray}
\int_{\Omega }\left\Vert X_{\omega }^{1}\right\Vert _{2-var}^{2}\mu \left(
d\omega \right) &\leq &72\sum_{n=0}^{\infty }\log _{2}\left( n+1\right)
\left\vert c_{n}\right\vert ^{2},
\label{inner summary of 2-rough path norm of long time behavior} \\
\int_{\Omega }\left\Vert A_{\omega }^{1}\right\Vert _{1-var}\mu \left(
d\omega \right) &\leq &170\sum_{n=0}^{\infty }\log _{2}\left( n+1\right)
\left\vert c_{n}\right\vert ^{2}.  \notag
\end{eqnarray}%
Therefore, if we can prove that for any $n\geq 1$, 
\begin{eqnarray}
\int_{\Omega }\left\Vert X_{\omega }\right\Vert _{2-var,\left[ 2^{n},2^{n+1}%
\right] }^{2}\mu \left( d\omega \right) &\leq
&4C\sum_{k=2^{n}+1}^{2^{n+1}}\log _{2}\left( k+1\right) \left\vert
c_{k}\right\vert ^{2},
\label{inner summary of 2-rough path norm of local behavior} \\
\int_{\Omega }\left\Vert A_{\omega }\right\Vert _{1-var,\left[ 2^{n},2^{n+1}%
\right] }\mu \left( d\omega \right) &\leq &2\left( C+1\right)
\sum_{k=2^{n}+1}^{2^{n+1}}\log _{2}\left( k+1\right) \left\vert
c_{k}\right\vert ^{2}.  \notag
\end{eqnarray}%
Then according to Lemma \ref{Lemma: key inequality for path and area} (on p%
\pageref{Lemma: key inequality for path and area}), 
\begin{eqnarray*}
&&\int_{\Omega }\left( \left\Vert X_{\omega }\right\Vert
_{2-var}^{2}+\left\Vert A_{\omega }\right\Vert _{1-var}\right) \mu \left(
d\omega \right) \\
&\leq &6\left\vert c_{1}\right\vert ^{2}+6\left\vert c_{2}\right\vert
^{2}+\int_{\Omega }\left( 6\left\Vert X_{\omega }^{1}\right\Vert
_{2-var}^{2}{}+\left\Vert A_{\omega }^{1}\right\Vert _{1-var}\right) \mu
\left( d\omega \right) \\
&&+\sum_{n=1}^{\infty }\int_{\Omega }\left( 6\left\Vert X_{\omega
}\right\Vert _{2-var,\left[ 2^{n},2^{n+1}\right] }^{2}+\left\Vert A_{\omega
}\right\Vert _{1-var,\left[ 2^{n},2^{n+1}\right] }\right) \mu \left( d\omega
\right) .
\end{eqnarray*}%
Substitute in $\left( \ref{inner summary of 2-rough path norm of long time
behavior}\right) $ and $\left( \ref{inner summary of 2-rough path norm of
local behavior}\right) $, we get 
\begin{equation*}
\int_{\Omega }\left( \left\Vert X_{\omega }\right\Vert
_{2-var}^{2}+\left\Vert A_{\omega }\right\Vert _{1-var}\right) \mu \left(
d\omega \right) \leq \left( 604+26C\right) \sum_{n=0}^{\infty }\log
_{2}\left( n+1\right) \left\vert c_{n}\right\vert ^{2},
\end{equation*}%
where $604+26C=6\times 72+170+24C+2\left( 1+C\right) $. Thus, if the two
inequalities in $\left( \ref{inner summary of 2-rough path norm of local
behavior}\right) $ are true, then $\left( X,A\right) $ is a geometric $2$%
-rough process under the condition $\sum_{n=0}^{\infty }\log _{2}\left(
n+1\right) \left\vert c_{n}\right\vert ^{2}<\infty $. Therefore, in the
following, we concentrate on two inequalities in $\left( \ref{inner summary
of 2-rough path norm of local behavior}\right) $.

Suppose we are working on $\left[ 2^{n},2^{n+1}\right] $ for some fixed
integer $n\geq 1$.

For any fixed finite partition $D=\left\{ \left[ m_{k},m_{k+1}\right]
\right\} _{k}$ of $\left[ 2^{n},2^{n+1}\right] $, denote the dyadic
intervals in $D$ as $\left\{ I_{j}\right\} $ (i.e. $\left[ m_{k},m_{k+1}%
\right] $ which are dyadic), denote the non-dyadic intervals in $D$ as $%
\left\{ J_{k}\right\} $. Use Lemma \ref{Lemma bisection partition} to bisect
non-dyadic intervals: every $J_{k}$ can be decomposed as disjoint union of $%
J_{k}^{1}$ and $J_{k}^{2}$, such that $J_{k}^{1}$ and $J_{k}^{2}$ are
intervals of positive length, and there exist two disjoint dyadic intervals $%
I_{k}^{1}$, $I_{k}^{2}$, satisfying $J_{k}^{i}\subseteq I_{k}^{i}$ and $%
\left\vert J_{k}^{i}\right\vert >\frac{1}{2}\left\vert I_{k}^{i}\right\vert $%
, $i=1,2$. As a result, when bisecting a set of \textit{disjoint} non-dyadic
intervals $\left\{ J_{k}\right\} $, in the set of related dyadic intervals $%
\left\{ I_{k}^{1},I_{k}^{2}\right\} $, each dyadic interval is counted at
most once. (Otherwise, there are two disjoint $J_{k}^{i}$ share the same
dyadic interval $I$, so there must be one $J_{k}^{i}$ satisfies $\left\vert
J_{k}^{i}\right\vert \leq \frac{1}{2}\left\vert I\right\vert $,
contradicting with the selection of $I$.) Denote $\left\Vert X\right\Vert
_{\infty ,I}:=\sup_{I^{\prime }\subseteq I}\left\Vert X\left( I^{\prime
}\right) \right\Vert $. Then, 
\begin{align}
& \sum_{\left[ m_{k},m_{k+1}\right] \in D}\left\Vert X_{\omega }\left( \left[
m_{k},m_{k+1}\right] \right) \right\Vert ^{2}=\sum_{k}\left\Vert X_{\omega
}\left( J_{k}\right) \right\Vert ^{2}+\sum_{j}\left\Vert X_{\omega }\left(
I_{j}\right) \right\Vert ^{2}  \label{inner sum over finite partition} \\
& \leq 2\sum_{k}\left( \left\Vert X_{\omega }\left( J_{k}^{1}\right)
\right\Vert ^{2}+\left\Vert X_{\omega }\left( J_{k}^{2}\right) \right\Vert
^{2}\right) +\sum_{j}\left\Vert X_{\omega }\left( I_{j}\right) \right\Vert
^{2}  \notag \\
& \leq 2\sum_{k}\left( \left\Vert X_{\omega }\right\Vert _{\infty
,I_{k}^{1}}^{2}+\left\Vert X_{\omega }\right\Vert _{\infty
,I_{k}^{2}}^{2}\right) +\sum_{j}\left\Vert X_{\omega }\right\Vert _{\infty
,I_{j}}^{2}\leq 2\sum_{I\in B_{[2^{n},2^{n+1}]}}\left\Vert X_{\omega
}\right\Vert _{\infty ,I}^{2},  \notag
\end{align}%
where we used that $I_{k}^{1}$, $I_{k}^{2}$ and $I_{j}$ are dyadic, and $%
\left\{ I_{k}^{1}\right\} \sqcup \left\{ I_{k}^{2}\right\} \sqcup \left\{
I_{j}\right\} \subseteq B_{\left[ 2^{n},2^{n+1}\right] }$. That $I_{k}^{i}$
are different as $k$ and $i$ vary, as we stated, is because $J_{k}^{i}$ are
disjoint, thus there can not be two $J_{k}^{i}$ share the same $I$; while $%
I_{k}^{i}$ differs from $I_{j}$ is because if $I_{k}^{i}=I_{j}$ for some $%
i,j,k$, then $J_{k}^{i}\subseteq I_{k}^{i}=I_{j}$, so $0<\left\vert
J_{k}^{i}\right\vert =\left\vert J_{k}^{i}\cap I_{j}\right\vert \leq
\left\vert J_{k}\cap I_{j}\right\vert $, contradicting with that $J_{k}$ and 
$I_{j}$ are disjoint since they are elements of finite partition $D$. Thus,
use $\left( \ref{inner sum over finite partition}\right) $ and take supremum
over all finite partitions of $\left[ 2^{n},2^{n+1}\right] $, we get,%
\begin{equation*}
\left\Vert X_{\omega }\right\Vert _{2-var,\left[ 2^{n},2^{n+1}\right]
}^{2}\leq 2\sum_{I\in B_{[2^{n},2^{n+1}]}}\left\Vert X_{\omega }\right\Vert
_{\infty ,I}^{2}.
\end{equation*}%
Using the assumption (Hardy property) that for any interval $I$, $%
\int_{\Omega }\left\Vert X_{\omega }\right\Vert _{\infty ,I}^{2}\mu \left(
d\omega \right) \leq C\int_{\Omega }\left\Vert X_{\omega }\left( I\right)
\right\Vert ^{2}\mu \left( d\omega \right) $ and $\left( \ref{Property II of
DI_J}\right) $, i.e. 
\begin{equation*}
\sum_{I\in B_{[2^{n},2^{n+1}]}}\int_{\Omega }\left\Vert X_{\omega }\left(
I\right) \right\Vert ^{2}\mu \left( d\omega \right) \leq 2\log _{2}\left(
2^{n}+1\right) \sum_{k=2^{n}+1}^{2^{n+1}}\left\vert c_{k}\right\vert ^{2},
\end{equation*}%
we get, for any integer $n$, 
\begin{eqnarray}
&&\int_{\Omega }\left\Vert X_{\omega }\right\Vert _{2-var,\left[
2^{n},2^{n+1}\right] }^{2}\mu \left( d\omega \right) \leq 2\int_{\Omega
}\sum_{I\in B_{[2^{n},2^{n+1}]}}\left\Vert X_{\omega }\right\Vert _{\infty
,I}^{2}\mu \left( d\omega \right)
\label{inner estimation of 2var of X on 2n 2n+1} \\
&\leq &2C\sum_{I\in B_{[2^{n},2^{n+1}]}}\int_{\Omega }\left\Vert X_{\omega
}\left( I\right) \right\Vert ^{2}\mu \left( d\omega \right) \leq 4C\log
_{2}\left( 2^{n}+1\right) \sum_{k=2^{n}+1}^{2^{n+1}}\left\vert
c_{k}\right\vert ^{2}  \notag \\
&\leq &4C\sum_{k=2^{n}+1}^{2^{n+1}}\log _{2}\left( k+1\right) \left\vert
c_{k}\right\vert ^{2}.  \notag
\end{eqnarray}

Then we estimate $1$-variation of $A_{\omega }$ on $\left[ 2^{n},2^{n+1}%
\right] $. On dyadic interval $I\subseteq \left[ 2^{n},2^{n+1}\right] $, use 
$\left( \ref{inner estimation of area on dyadic interal}\right) $, we have 
\begin{equation}
\left\Vert A_{\omega }\left( I\right) \right\Vert \leq \sum_{I^{\prime }\in
B_{I}\backslash \left\{ I\right\} }\left\Vert X_{\omega }\left( I^{\prime
}\right) \right\Vert ^{2}.
\label{inner estimation of area on dyadic intervals}
\end{equation}%
Suppose $J\subseteq \left[ 2^{n},2^{n+1}\right] $ is a non-dyadic interval.
Use Lemma \ref{Lemma bisection partition} to bisect $J=J^{1}\cup J^{2}$,
with associated dyadic intervals $I^{i}$, then $\left\vert I^{i}\cap
J\right\vert =\left\vert J^{i}\right\vert >\frac{1}{2}\left\vert
I^{i}\right\vert $. Thus $I^{i}\in \widetilde{B}_{J}$ ($\widetilde{B}_{J}$
is defined at $\left( \ref{Definition of B_J_tile}\right) $), and 
\begin{eqnarray*}
\left\Vert A_{\omega }\left( J\right) \right\Vert &\leq &\left\Vert
A_{\omega }\left( J^{1}\right) \right\Vert +\left\Vert A_{\omega }\left(
J^{2}\right) \right\Vert +\left\Vert X_{\omega }\left( J^{1}\right)
\right\Vert \left\Vert X_{\omega }\left( J^{2}\right) \right\Vert \\
&\leq &\left\Vert A_{\omega }\left( J^{1}\right) \right\Vert +\left\Vert
A_{\omega }\left( J^{2}\right) \right\Vert +\left\Vert X_{\omega
}\right\Vert _{\infty ,I^{1}}^{2}+\left\Vert X_{\omega }\right\Vert _{\infty
,I^{2}}^{2}.
\end{eqnarray*}%
The bisecting process terminates if both $J^{1}$ and $J^{2}$ are dyadic,
otherwise, continue to bisect non-dyadic $J^{1}$ and/or $J^{2}$, so on and
so forth, until all the intervals left are dyadic. According to Lemma \ref%
{Lemma repeatedly bisecting non-dyadic interval got the same dyadic
partition}, all the dyadic intervals left constitute the dyadic partition of 
$J$ in Lemma \ref{Lemma of dyadic partition}.

The dyadic intervals, which are by-products of our sequence of bisections
(e.g. $I^{1}$ and $I^{2}$ from bisecting $J$), are elements of $\widetilde{B}%
_{J}$, because if dyadic interval $I$ is obtained from bisecting interval $%
J^{\prime }\subseteq J$, then $I\in $ $\widetilde{B}_{J^{\prime }}\subseteq 
\widetilde{B}_{J}$ ($I\in $ $\widetilde{B}_{J^{\prime }}$ is the same reason
as $I^{1}$, $I^{2}\in \widetilde{B}_{J}$; $\widetilde{B}_{J^{\prime
}}\subseteq \widetilde{B}_{J}$ is $(iii)$ on p$\pageref{Third property of
B_J_tile}$). Moreover, these by-product dyadic intervals differ from one
another. Otherwise, suppose $J^{\left( 1\right) }$ and $J^{\left( 2\right) }$
are two different intervals generated in the bisecting process, sharing the
same dyadic interval $I$, i.e. $J^{\left( i\right) }\subseteq I$, and $%
\left\vert J^{\left( i\right) }\right\vert >\frac{1}{2}\left\vert
I\right\vert $, then $\left\vert J^{\left( 1\right) }\cap J^{\left( 2\right)
}\right\vert >0$, and $I$ is the smallest dyadic interval which includes $%
J^{\left( 1\right) }$($J^{\left( 2\right) }$). Since $J^{\left( 1\right) }$
and $J^{\left( 2\right) }$ are sub-intervals generated in the process of
decomposing $J$, so if $\left\vert J^{\left( 1\right) }\cap J^{\left(
2\right) }\right\vert >0$, then one is wholly included in another. Thus,
without loss of generality, suppose $J^{\left( 2\right) }\subset J^{\left(
1\right) }$, then $J^{\left( 2\right) }$ is obtained from further bisecting $%
J^{\left( 1\right) }$. When bisecting $J^{\left( 1\right) }$, according to
Lemma \ref{Lemma bisection partition}, there exist two disjoint dyadic
intervals $I^{\prime }$ and $I^{\prime \prime }$, s.t. $\left\vert J^{\left(
1\right) }\cap I^{\prime }\right\vert >0$, $\left\vert J^{\left( 1\right)
}\cap I^{\prime \prime }\right\vert >0$. Since $J^{\left( 2\right) }$ is
obtained from further bisecting $J^{\left( 1\right) }$, without loss of
generality, we assume $J^{\left( 2\right) }\subseteq I^{\prime }$. As we
denoted, $I$ is the smallest dyadic interval containing $J^{\left( 2\right)
} $, so $I\subseteq I^{\prime }$, while $I$ is also the smallest dyadic
interval containing $J^{\left( 1\right) }$, so $J^{\left( 1\right)
}\subseteq I^{\prime }$, contradictory with that $I^{\prime }$ and $%
I^{\prime \prime }$ are disjoint and $\left\vert J^{\left( 1\right) }\cap
I^{\prime \prime }\right\vert >0$.

As a result, if denote the dyadic partition of $J$ in Lemma \ref{Lemma of
dyadic partition} as $\cup _{k}I_{k}$, use the estimation for $A_{\omega }$
on dyadic intervals (i.e.$\left( \ref{inner estimation of area on dyadic
interal}\right) $), we get%
\begin{equation*}
\sum_{k}\left\Vert A_{\omega }\left( I_{k}\right) \right\Vert \leq
\sum_{k}\sum_{I\in B_{I_{k}}\backslash \left\{ I_{k}\right\} }\left\Vert
X_{\omega }\left( I\right) \right\Vert ^{2}\leq \sum_{I\in B_{J}}\left\Vert
X_{\omega }\left( I\right) \right\Vert ^{2}\text{.}
\end{equation*}%
Thus (all by-products dyadic intervals are elements of $\widetilde{B}_{J}$,
and they are different from one another), 
\begin{equation}
\left\Vert A_{\omega }\left( J\right) \right\Vert \!\leq
\!\sum_{k}\left\Vert A_{\omega }\left( I_{k}\right) \right\Vert
\!+\!\sum_{I\in \widetilde{B}_{J}.}\left\Vert X_{\omega }\right\Vert
_{\infty ,I}^{2}\!\leq \!\sum_{I\in B_{J}}\left\Vert X_{\omega }\left(
I\right) \right\Vert ^{2}\!+\!\sum_{I\in \widetilde{B}_{J}}\left\Vert
X_{\omega }\right\Vert _{\infty ,I}^{2}.
\label{inner estimation of area on non-dyadic intervals}
\end{equation}

Therefore, suppose $\left\{ I_{j}\right\} _{j}\cup \left\{ J_{k}\right\}
_{k} $ is a finite partition of $\left[ 2^{n},2^{n+1}\right] $, with $I_{j}$
dyadic intervals and $J_{k}$ non-dyadic intervals. Combine estimation on
dyadic intervals in $\left( \ref{inner estimation of area on dyadic
intervals}\right) $ and on non-dyadic intervals in $\left( \ref{inner
estimation of area on non-dyadic intervals}\right) $, we have%
\begin{eqnarray*}
&&\sum_{j}\left\Vert A_{\omega }\left( I_{j}\right) \right\Vert
+\sum_{k}\left\Vert A_{\omega }\left( J_{k}\right) \right\Vert \\
&\leq &\sum_{j}\sum_{I\in B_{I_{j}}}\left\Vert X_{\omega }\left( I\right)
\right\Vert ^{2}+\sum_{k}(\sum_{I\in B_{J_{k}}}\left\Vert X_{\omega }\left(
I\right) \right\Vert ^{2}+\sum_{I\in \widetilde{B}_{J_{k}}}\left\Vert
X_{\omega }\right\Vert _{\infty ,I}^{2})\text{.}
\end{eqnarray*}%
Using $\left( \sqcup _{j}B_{I_{j}}\right) \sqcup \left( \sqcup
_{k}B_{J_{k}}\right) \subseteq B_{[2^{n},2^{n+1}]}$ (according to $\left( %
\ref{Property 1 of DI_J}\right) $), $\sqcup _{k}\widetilde{B}%
_{J_{k}}\subseteq \widetilde{B}_{[2^{n},2^{n+1}]}$ (according to $\left(
iv\right) $ on p$\pageref{Fourth property of B_J_Tile}$), and $%
B_{[2^{n},2^{n+1}]}=\widetilde{B}_{[2^{n},2^{n+1}]}$ for dyadic interval $%
\left[ 2^{n},2^{n+1}\right] $ (according to $\left( ii\right) $ on p$\pageref%
{Second property of B_J_tile}$), we get%
\begin{equation*}
\left\Vert A_{\omega }\right\Vert _{1-var,\left[ 2^{n},2^{n+1}\right] }\leq
\sum_{I\in B_{[2^{n},2^{n+1}]}}\left( \left\Vert X_{\omega }\left( I\right)
\right\Vert ^{2}+\left\Vert X_{\omega }\right\Vert _{\infty ,I}^{2}\right) .
\end{equation*}%
Integrate both sides, use $\int_{\Omega }\left\Vert X\right\Vert _{\infty
,I}^{2}\mu \left( d\omega \right) \leq C\int_{\Omega }\left\Vert X\left(
I\right) \right\Vert ^{2}\mu \left( d\omega \right) $, and $\left( \ref%
{Property II of DI_J}\right) $, i.e.%
\begin{equation*}
\sum_{I\in B_{[2^{n},2^{n+1}]}}\int_{\Omega }\left\Vert X_{\omega }\left(
I\right) \right\Vert ^{2}\mu \left( d\omega \right) \leq 2\log _{2}\left(
2^{n}+1\right) \sum_{k=2^{n}+1}^{2^{n+1}}\left\vert c_{k}\right\vert ^{2},
\end{equation*}%
we get, for any $n\geq 1$,%
\begin{align}
\int_{\Omega }\left\Vert A_{\omega }\right\Vert _{1-var,\left[ 2^{n},2^{n+1}%
\right] }\mu \left( d\omega \right) & \leq \left( 1+C\right) \sum_{I\in
B_{[2^{n},2^{n+1}]}}\int_{\Omega }\left\Vert X_{\omega }\left( I\right)
\right\Vert ^{2}\mu \left( d\omega \right)
\label{inner estimation of 1var of area on 2n 2n+1} \\
& \leq 2\left( 1+C\right) \sum_{k=2^{n}+1}^{2^{n+1}}\log _{2}\left(
k+1\right) \left\vert c_{k}\right\vert ^{2}.  \notag
\end{align}%
Combined with reasoning at the beginning of the proof and $\left( \ref{inner
estimation of 2var of X on 2n 2n+1}\right) $, proof finishes.
\end{proof}

\section{Sobolev spaces $H_{Log}^{s}$}

In this section, we identify an equivalent norm on the space of functions
whose Fourier coefficients satisfy $\sum_{n}\left( \log _{2}\left(
n+1\right) \right) ^{2s}\left\vert c_{n}\right\vert ^{2}<\infty $ for some $%
s>0$. We also construct an example to demonstrate that, the condition $%
\sum_{n}w\left( n\right) \left\vert c_{n}\right\vert ^{2}<\infty $ is not
necessary for the partial sum process of $L^{2}$ Fourier series to be a
geometric $2$-rough process, for any Weyl multiplier $\left\{ w\left(
n\right) \right\} $ increasing strictly faster than $\{\left( \log _{2}\log
_{2}n\right) ^{2}\}$.

Let $H^{\delta }$ be the sobolev space $W^{\delta ,2}$. The fact that $f:%
\left[ -\pi ,\pi \right] \rightarrow 
\mathbb{R}
^{d}$ belongs to $H^{\delta }$ for some $0<\delta <1$, can be stated
equivalently in the following two ways (Theorem $8.5$ in \cite{R. Kress}): 
\begin{equation}
\sum_{n=0}^{\infty }n^{2\delta }\left\vert c_{n}\right\vert ^{2}<\infty ,
\label{First definition of sobolev space}
\end{equation}%
and 
\begin{equation}
\text{ }\int_{-\pi }^{\pi }\int_{-\pi }^{\pi }\frac{\left\vert f\left(
u\right) -f\left( v\right) \right\vert ^{2}}{\left\vert \sin \frac{u-v}{2}%
\right\vert ^{2\delta +1}}dudv<\infty ,
\label{Second definition of sobolev space}
\end{equation}%
where $\left\{ c_{n}\right\} $ are the Fourier coefficients of $f$ (suppose $%
f=\left( f_{1},f_{2},\dots ,f_{d}\right) $, then $c_{n}=\left(
c_{n}^{1},c_{n}^{2},\dots ,c_{n}^{d}\right) \in 
\mathbb{R}
^{2d}$, with $c_{n}^{k}=\int_{-\pi }^{\pi }f_{k}\left( \theta \right)
e^{in\theta }d\theta $). When $\delta =0$, the space defined by $\left( \ref%
{Second definition of sobolev space}\right) $ is strictly included in $L^{2}$%
, which, as we will prove (also proved in Thm4 \cite{W. Beckner}), is
equivalent to%
\begin{equation*}
\sum_{n=0}^{\infty }\log _{2}\left( n+1\right) \left\vert c_{n}\right\vert
^{2}<\infty .
\end{equation*}

To fit the framework of our theorems,

\begin{definition}
Define sobolev spaces $H_{Log}^{s}$, $-\infty <s<\infty $, as the linear
space of $%
\mathbb{R}
^{d}$ valued functions on $\left[ -\pi ,\pi \right] $ with finite the
following norm:%
\begin{equation}
\left\Vert f\right\Vert _{Log,s}:=\left( \sum_{n=0}^{\infty }\left( \log
_{2}\left( n+1\right) \right) ^{2s}\left\vert c_{n}\right\vert ^{2}\right) ^{%
\frac{1}{2}},  \label{First definition of Log sobolev space}
\end{equation}%
where $\left\{ c_{n}\right\} $ are Fourier coefficients of $f$.
\end{definition}

Similar to $H^{s}$, $H_{Log}^{s}$ is a separable Hilbert space for any $%
-\infty <s<\infty $, with trigonometric polynomials as a dense subset; When $%
0\leq s<\infty $, $H_{Log}^{-s}$ is the dual of $H_{Log}^{s}$ in $L^{2}$;
and $H_{Log}^{q}$ can be compactly embedded into $H_{Log}^{p}$ for any $q>p$%
. Moreover, for the interpolation space $\left(
H_{Log}^{p},H_{Log}^{q}\right) _{\theta ,2}=H_{Log}^{r}$, where $r=\left(
1-\theta \right) p+\theta q$, H\"{o}lder inequality holds: 
\begin{equation*}
\left\Vert f\right\Vert _{Log,r}\leq \left\Vert f\right\Vert
_{Log,p}^{1-\theta }\left\Vert f\right\Vert _{Log,q}^{\theta }.
\end{equation*}%
All these properties can be proved as counterparts as those of $H^{\delta }$
(e.g. p108-p117, \cite{R. Kress}).

The function 
\begin{equation*}
f_{s,\epsilon }\left( x\right) =\frac{1}{x^{\frac{1}{2}}\left\vert \log _{2}%
\frac{x}{2}\right\vert ^{s+\frac{1}{2}}\left\vert \log _{2}\left(
2\left\vert \log _{2}\frac{x}{2}\right\vert \right) \right\vert ^{\frac{1}{2}%
+\epsilon }}\text{, \ }x\in \left( 0,1\right) \text{,}
\end{equation*}%
(according to Theorem $2.24$ on p190 in Vol I \cite{Zygmund}) belongs to $%
H_{Log}^{s}$ when $\epsilon >0$, not belongs to $H_{Log}^{s}$ when $\epsilon
\leq 0$.

Next, we prove that there exists an equivalent norm on $H_{Log}^{s}$ as the
one for $H^{s}$ in $\left( \ref{Second definition of sobolev space}\right) $%
, which is inspired by Theorem $8.5$ in \cite{R. Kress}.\ (When $s=\frac{1}{2%
}$, the equivalency is proved in Thm4 \cite{W. Beckner}.)

Before that, we prove a lemma.

\begin{lemma}
\label{Lemma simplify norm for trigonometric monomials}Suppose $s\in \left(
-\infty ,\infty \right) $. For $n\in 
\mathbb{N}
$, if denote 
\begin{gather*}
T^{s}\left( n\right) :=\int_{-\pi }^{\pi }\int_{-\pi }^{\pi }\frac{%
\left\vert \sin \left( \frac{1}{2}n\left( u-v\right) \right) \right\vert ^{2}%
}{\left\vert \sin \frac{u-v}{2}\right\vert }(\log _{2}\frac{\pi }{\left\vert
\sin \frac{u-v}{2}\right\vert })^{2s-1}dudv\text{,} \\
\text{and \ }R^{s}\left( n\right) :=\int_{0}^{1}\frac{\left\vert \sin \left( 
\frac{1}{2}\pi nt\right) \right\vert ^{2}}{t}(\log _{2}\frac{2}{t})^{2s-1}dt%
\text{,}
\end{gather*}%
then there exists $0<c_{s}\leq C_{s}<\infty $ such that 
\begin{equation}
c_{s}\,R^{s}\!\left( n\right) \leq T^{s}\!\left( n\right) \leq
C_{s}\,R^{s}\!\left( n\right) \text{, }\forall n\in 
\mathbb{N}
\text{.}  \label{inequality between two expressions}
\end{equation}
\end{lemma}

\begin{proof}
Denote $\xi :=\frac{u+v}{2}$, $\eta :=\frac{u-v}{2}$, then 
\begin{equation*}
T^{s}\left( n\right) =8\int_{0}^{\pi }\int_{0}^{\pi -\eta }\frac{\left\vert
\sin n\eta \right\vert ^{2}}{\sin \eta }(\log _{2}\frac{\pi }{\sin \eta }%
)^{2s-1}d\xi d\eta .
\end{equation*}%
Since 
\begin{multline*}
\left\{ \left( \eta ,\xi \right) |0\leq \eta \leq \frac{\pi }{2},0\leq \xi
\leq \frac{\pi }{2}\right\} \subset \left\{ \left( \eta ,\xi \right) |0\leq
\eta \leq \pi ,0\leq \xi \leq \pi -\eta \right\} \\
\subset \left\{ \left( \eta ,\xi \right) |0\leq \eta \leq \pi ,0\leq \xi
\leq \pi \right\} ,
\end{multline*}%
we have%
\begin{equation}
4\pi \int_{0}^{\frac{\pi }{2}}\frac{\left\vert \sin \left( nt\right)
\right\vert ^{2}}{\sin t}(\log _{2}\frac{\pi }{\sin t})^{2s-1}dt\leq
T^{s}\left( n\right) \leq 16\pi \int_{0}^{\frac{\pi }{2}}\frac{\left\vert
\sin \left( nt\right) \right\vert ^{2}}{\sin t}(\log _{2}\frac{\pi }{\sin t}%
)^{2s-1}dt.  \label{inner inequality between Tn and intermediate}
\end{equation}%
Then by using the inequality 
\begin{equation*}
\frac{2}{\pi }t\leq \sin t\leq t\text{, \ }t\in \left[ 0,\frac{\pi }{2}%
\right] \text{,}
\end{equation*}%
one can prove that, there exists constant $0<b_{s}\leq B_{s}<\infty $, s.t.
(with $R^{s}\left( n\right) $ defined in the statement of this lemma)%
\begin{equation}
b_{s}\,R^{s}\!\left( n\right) \leq \int_{0}^{\frac{\pi }{2}}\frac{\left\vert
\sin \left( nt\right) \right\vert ^{2}}{\sin t}(\log _{2}\frac{\pi }{\sin t}%
)^{2s-1}dt\leq B_{s}\,R^{s}\!\left( n\right) \text{, \ }\forall n\in 
\mathbb{N}
\text{.}  \label{inner inequality between Rn and intermediate}
\end{equation}%
Combine $\left( \ref{inner inequality between Tn and intermediate}\right) $
and $\left( \ref{inner inequality between Rn and intermediate}\right) $,
lemma holds.
\end{proof}

In the following theorem, we use Euclidean norm, so that $k_{s}$ and
\thinspace $K_{s}$ are independent of dimension $d$.

\medskip

\noindent%
\textbf{Theorem \ref{Theorem equivalent norm on Fourier series}} \ \textit{%
For }$0<s<\infty $\textit{, there exist constants }$0<k_{s}\leq K_{s}<\infty 
$\textit{, such that for }any\textit{\ }$f\in L^{2}\left( \left[ -\pi ,\pi %
\right] ,%
\mathbb{R}
^{d}\right) $\textit{\ with Fourier coefficients }$\left\{ c_{n}\right\} $%
\textit{, }%
\begin{gather}
\text{if denote }L\!\left( f\right) :=\int_{-\pi }^{\pi }\int_{-\pi }^{\pi }%
\frac{\left\vert f\left( u\right) -f\left( v\right) \right\vert ^{2}}{%
\left\vert \sin \frac{u-v}{2}\right\vert }(\log _{2}\frac{\pi }{\left\vert
\sin \frac{u-v}{2}\right\vert })^{2s-1}dudv\text{ }
\label{Second definition of logaritham sobolev space} \\
\text{and }l\!\left( f\right) :=\sum_{n=0}^{\infty }\left( \log _{2}\left(
n+1\right) \right) ^{2s}\left\vert c_{n}\right\vert ^{2}\text{, then }%
k_{s}\,l\!\left( f\right) \leq L\!\left( f\right) \leq K_{s}\,l\!\left(
f\right) .  \notag
\end{gather}

\medskip

\begin{proof}
Fix $s>0$. Without loss of generality, we assume $f$ is one-dimensional.
Since trigonometric polynomials are dense in $H_{Log}^{s}$, we only prove
the theorem for trigonometric polynomials. It can be verified that $e^{inx}$%
, $n\in 
\mathbb{Z}
$, are orthogonal w.r.t. this inner product:%
\begin{equation*}
\left\langle f_{1},f_{2}\right\rangle =\int_{-\pi }^{\pi }\int_{-\pi }^{\pi }%
\frac{\func{Re}\left( \left( f_{1}\left( u\right) -f_{1}\left( v\right)
\right) \overline{\left( f_{2}\left( u\right) -f_{2}\left( v\right) \right) }%
\right) }{\left\vert \sin \frac{u-v}{2}\right\vert }(\log _{2}\frac{\pi }{%
\left\vert \sin \frac{u-v}{2}\right\vert })^{2s-1}dudv.
\end{equation*}%
Thus, for any trigonometric polynomial $f_{N}\left( \theta \right)
:=\sum_{n=-N}^{N}c_{n}e^{in\theta }$, we have%
\begin{equation*}
L\left( f_{N}\right) =\sum_{n=-N}^{N}\left\vert c_{n}\right\vert ^{2}L\left(
e^{in\cdot }\right) \text{.}
\end{equation*}%
Since $L\left( 1\right) =0$, and when $n\geq 1$,%
\begin{equation*}
L\left( e^{in\cdot }\right) =L\left( e^{-in\cdot }\right) =4\int_{-\pi
}^{\pi }\int_{-\pi }^{\pi }\frac{\left\vert \sin \frac{n}{2}\left(
u-v\right) \right\vert ^{2}}{\left\vert \sin \frac{u-v}{2}\right\vert }(\log
_{2}\frac{\pi }{\left\vert \sin \frac{u-v}{2}\right\vert })^{2s-1}dudv\text{,%
}
\end{equation*}%
based on Lemma \ref{Lemma simplify norm for trigonometric monomials}, the
problem boils down to: for any $s\in \left( 0,\infty \right) $, there exists
integer $N_{s}$ and constants $0<b_{s}\leq B_{s}<\infty $, s.t. for any $%
n\geq N_{s}$,%
\begin{equation*}
b_{s}\left( \log _{2}\left( \pi n\right) \right) ^{2s}\leq R^{s}\left(
n\right) :=\int_{0}^{1}\frac{\sin ^{2}\left( \frac{1}{2}\pi nt\right) }{t}%
\left\vert \log _{2}\frac{t}{2}\right\vert ^{2s-1}dt\leq B_{s}\left( \log
_{2}\left( \pi n\right) \right) ^{2s}\text{.}
\end{equation*}

Denote%
\begin{equation*}
R^{s}\left( n\right) =\int_{0}^{\pi n}\frac{\sin ^{2}\frac{1}{2}t}{t}%
\left\vert \log _{2}\frac{t}{2}-\log _{2}\left( \pi n\right) \right\vert
^{2s-1}dt=\int_{0}^{2}+\int_{2}^{\pi n}:=R_{1}^{s}\left( n\right)
+R_{2}^{s}\left( n\right) .
\end{equation*}

For $R_{1}^{s}\left( n\right) $, 
\begin{equation*}
\frac{R_{1}^{s}\left( n\right) }{\left( \log _{2}\left( \pi n\right) \right)
^{2s-1}}=\int_{0}^{2}\frac{\sin ^{2}\frac{1}{2}t}{t}\left\vert \frac{1}{\log
_{2}\left( \pi n\right) }\log _{2}\frac{t}{2}-1\right\vert ^{2s-1}dt.
\end{equation*}%
When $n\geq 1$, 
\begin{equation*}
1\leq \left\vert \frac{1}{\log _{2}\left( \pi n\right) }\log _{2}\frac{t}{2}%
-1\right\vert \leq 1+\left\vert \log _{2}\frac{t}{2}\right\vert \text{, }%
t\in \left( 0,2\right) .
\end{equation*}%
Thus when $s\geq \frac{1}{2}$, 
\begin{equation}
0<\int_{0}^{2}\frac{\sin ^{2}\frac{1}{2}t}{t}dt\leq \frac{R_{1}^{s}\left(
n\right) }{\left( \log _{2}\left( \pi n\right) \right) ^{2s-1}}\leq
\int_{0}^{2}\frac{\sin ^{2}\frac{1}{2}t}{t}\left( 1+\left\vert \log _{2}%
\frac{t}{2}\right\vert \right) ^{2s-1}dt<\infty .
\label{inner upper and lower bounds}
\end{equation}%
When $0<s<\frac{1}{2}$, the upper bound and lower bound in $\left( \ref%
{inner upper and lower bounds}\right) $ exchange. Thus, $R_{1}^{s}\left(
n\right) \sim \left( \log _{2}\left( \pi n\right) \right) ^{2s-1}$, and for
any $\epsilon >0$, there exists $N_{\epsilon }\geq 1$, s.t.%
\begin{equation}
\left\vert R_{1}^{s}\left( n\right) \right\vert \leq \epsilon \left( \log
_{2}\left( \pi n\right) \right) ^{2s}\text{, }\forall n\geq N_{\epsilon }%
\text{.}  \label{inner upper and lower bounds for R1s(n)}
\end{equation}

For $R_{2}^{s}\left( n\right) $, 
\begin{equation}
\frac{R_{2}^{s}\left( n\right) }{\left( \log _{2}\left( \pi n\right) \right)
^{2s}}=\frac{1}{\log _{2}\left( \pi n\right) }\int_{2}^{\pi n}\frac{\sin ^{2}%
\frac{1}{2}t}{t}\left\vert \frac{1}{\log _{2}\left( \pi n\right) }\log _{2}%
\frac{t}{2}-1\right\vert ^{2s-1}dt.  \label{inner representation of R2s(n)}
\end{equation}%
For lower bound: When $2\leq t\leq \sqrt{n}\pi $,%
\begin{equation*}
0\leq \frac{1}{\log _{2}\left( \pi n\right) }\log _{2}\frac{t}{2}\leq \frac{1%
}{\log _{2}\left( \pi n\right) }\left( \frac{1}{2}\log _{2}n+\log _{2}\frac{%
\pi }{2}\right) \leq \frac{1}{2},
\end{equation*}%
so%
\begin{equation*}
\frac{1}{2}\leq \left\vert \frac{1}{\log _{2}\left( \pi n\right) }\log _{2}%
\frac{t}{2}-1\right\vert \leq 1\text{ when }2\leq t\leq \sqrt{n}\pi \text{.}
\end{equation*}%
Denote $[\sqrt{n}$ $]$ as the integer part of $\sqrt{n}$. Then when $s\geq 
\frac{1}{2}$, $n\geq 1$,%
\begin{equation*}
\frac{R_{2}^{s}\left( n\right) }{\left( \log _{2}\left( \pi n\right) \right)
^{2s}}\geq \frac{1}{2^{2s-1}\log _{2}\left( \pi n\right) }\sum_{k=1}^{[\sqrt{%
n}\text{ }]-1}\int_{k\pi }^{\left( k+1\right) \pi }\frac{\sin ^{2}\frac{1}{2}%
t}{t}dt\geq \frac{1}{2^{2s}\log _{2}\left( \pi n\right) }\sum_{k=1}^{[\sqrt{n%
}\text{ }]-1}\frac{1}{k+1}.
\end{equation*}%
\begin{equation*}
\text{While \ }\sum_{k=1}^{[\sqrt{n}\text{ }]-1}\frac{1}{k+1}=\sum_{k=1}^{[%
\sqrt{n}\text{ }]}\frac{1}{k}-1\geq \int_{1}^{[\sqrt{n}\text{ }]+1}\frac{1}{x%
}dx-1=\ln \left( [\sqrt{n}\text{ }]+1\right) -1\geq \frac{1}{2}\ln \left(
n\right) -1\text{.}
\end{equation*}%
Thus, for $s\geq \frac{1}{2}$, when $n\geq \left[ e^{4}\pi \right] +1$, we
have $\frac{\ln n-2}{\ln n+\ln \pi }\geq \frac{1}{2}$, and%
\begin{equation*}
\frac{R_{2}^{s}\left( n\right) }{\left( \log _{2}\left( \pi n\right) \right)
^{2s}}\geq \frac{\ln 2\left( \ln n-2\right) }{2^{2s+1}\left( \ln n+\ln \pi
\right) }\geq \frac{\ln 2}{2^{2s+2}}\text{.}
\end{equation*}%
Similarly, for $0<s<\frac{1}{2}$, when $n\geq \left[ e^{4}\pi \right] +1$,
we have%
\begin{equation*}
\frac{R_{2}^{s}\left( n\right) }{\left( \log _{2}\left( \pi n\right) \right)
^{2s}}\geq \frac{\ln 2}{8}.
\end{equation*}%
\ 

\noindent For the upper bound of $\frac{R_{2}^{s}\left( n\right) }{\left(
\log _{2}\left( \pi n\right) \right) ^{2s}}$, in $\left( \ref{inner
representation of R2s(n)}\right) $ let $y=\frac{\log _{2}\frac{t}{2}}{\log
_{2}\left( \pi n\right) }$, then%
\begin{equation*}
\frac{R_{2}^{s}\left( n\right) }{\left( \log _{2}\left( \pi n\right) \right)
^{2s}}\leq \ln 2\int_{0}^{1}\sin ^{2}\left( \left( \pi n\right) ^{y}\right)
\left( 1-y\right) ^{2s-1}dy\leq \ln 2\int_{0}^{1}\left( 1-y\right)
^{2s-1}dy_{1}=\frac{\ln 2}{2s}.
\end{equation*}%
Thus, when $n\geq \left[ e^{4}\pi \right] +1$,%
\begin{equation*}
\frac{\ln 2}{2^{2(s\vee \frac{1}{2})+2}}=\min \{\frac{\ln 2}{2^{2s+2}},\frac{%
\ln 2}{8}\}\leq \frac{R_{2}^{s}\left( n\right) }{\left( \log _{2}\left( \pi
n\right) \right) ^{2s}}\leq \frac{\ln 2}{2s}\text{.}
\end{equation*}

Therefore, if for $s>0$ let $\epsilon \left( s\right) =\frac{\ln 2}{%
2^{2(s\vee \frac{1}{2})+3}}$, then according to $\left( \ref{inner upper and
lower bounds for R1s(n)}\right) $, there exists integer $N_{\epsilon \left(
s\right) }\geq 1$, s.t. for any $n\geq N_{\epsilon \left( s\right) }$, $%
\left\vert R_{1}^{s}\left( n\right) \right\vert \leq \epsilon \left(
s\right) \left( \log _{2}\left( \pi n\right) \right) ^{2s}$. As a result, we
get: for any $n\geq N_{s}:=\max \left\{ N_{\epsilon \left( s\right) },\left[
e^{4}\pi \right] +1\right\} $,%
\begin{equation*}
\frac{\ln 2}{2^{2(s\vee \frac{1}{2})+3}}\leq \frac{R^{s}\left( n\right) }{%
\left( \log _{2}\left( \pi n\right) \right) ^{2s}}\leq \frac{\ln 2}{2s}+%
\frac{\ln 2}{2^{2(s\vee \frac{1}{2})+3}}\text{,}
\end{equation*}%
where we used $R_{2}^{s}\left( n\right) -\left\vert R_{1}^{s}\left( n\right)
\right\vert \leq R^{s}\left( n\right) \leq R_{2}^{s}\left( n\right)
+\left\vert R_{1}^{s}\left( n\right) \right\vert $. Combined with reasoning
at the beginning of the proof, proof finishes.
\end{proof}

\begin{remark}
\label{Remark equivalent identification of L2 functions}In similar way, one
can prove the equality that, for any $f\in L^{2}\left( \left[ -\pi ,\pi %
\right] ,%
\mathbb{R}
^{d}\right) $ (using Euclidean norm) 
\begin{equation*}
\int_{-\pi }^{\pi }\left\vert f\left( \theta \right) \right\vert ^{2}d\theta
=\frac{1}{2\pi }\left\vert \int_{-\pi }^{\pi }f\left( \theta \right) d\theta
\right\vert ^{2}+\frac{1}{4\pi }\int_{-\pi }^{\pi }\int_{-\pi }^{\pi
}\left\vert f\left( u\right) -f\left( v\right) \right\vert ^{2}dudv.
\end{equation*}%
Then $\int_{-\pi }^{\pi }\int_{-\pi }^{\pi }\left\vert f\left( u\right)
-f\left( v\right) \right\vert ^{2}dudv<\infty $ iff $f$ is in $L^{2}\left( %
\left[ -\pi ,\pi \right] ,%
\mathbb{R}
^{d}\right) $. Since $\left\vert \sin \frac{u-v}{2}\right\vert \leq 1$, from
this perspective, one can also get that 
\begin{equation*}
\int_{-\pi }^{\pi }\int_{-\pi }^{\pi }\frac{\left\vert f\left( u\right)
-f\left( v\right) \right\vert ^{2}}{\left\vert \sin \frac{u-v}{2}\right\vert 
}dudv<\infty \Longrightarrow f\text{ is an }L^{2}\text{ function.}
\end{equation*}
\end{remark}

\bigskip

Combine Theorem \ref{Theorem equivalent norm on Fourier series} (as proved
above) with Corollary \ref{Corollary for geometric rough process of Fourier
series} (on p\pageref{Corollary for geometric rough process of Fourier
series}), we get that if 
\begin{equation*}
\int_{-\pi }^{\pi }\int_{-\pi }^{\pi }\frac{\left\vert f\left( u\right)
-f\left( v\right) \right\vert ^{2}}{\left\vert \sin \frac{u-v}{2}\right\vert 
}dudv<\infty ,
\end{equation*}%
then $f$ is in $L^{2}$ (also Remark \ref{Remark equivalent identification of
L2 functions}), and the partial sum process of Fourier series of $f$ is a
geometric $2$-rough process (denoted as $\mathbf{X}$). Moreover, there
exists absolute constant $C$, s.t.%
\begin{equation}
\int_{-\pi }^{\pi }\left\Vert \mathbf{X}\left( \theta \right) \right\Vert
_{G^{\left( 2\right) }}^{2}d\theta \leq C\int_{-\pi }^{\pi }\int_{-\pi
}^{\pi }\frac{\left\vert f\left( u\right) -f\left( v\right) \right\vert ^{2}%
}{\left\vert \sin \frac{u-v}{2}\right\vert }dudv\sim \sum_{n=0}^{\infty
}\log _{2}\left( n+1\right) \left\vert c_{n}\right\vert ^{2}
\label{inequality integration}
\end{equation}%
However, although in $\left( \ref{inequality integration}\right) $ $\log
_{2}\left( n+1\right) $ can not be replaced by $o\left( \log _{2}\left(
n+1\right) \right) $, as we demonstrate below, the Weyl multiplier $\left\{
\log _{2}\left( n+1\right) \right\} $ is not necessary for the partial sum
process of Fourier series to be a geometric $2$-rough process (i.e. an
almost everywhere finite random variable with infinite expectation).

Before proceeding to the example, we give a lemma, which is all we need for
the example.

\begin{lemma}
\label{Lemma estimation for fixed theta}For $\theta \in \left( 0,2\pi
\right) $ and $n\geq 1$, if we define $Y_{\theta }^{n}:%
\mathbb{N}
\rightarrow 
\mathbb{C}
$ as 
\begin{gather}
Y_{\theta }^{n}\left( k\right) =\left\{ 
\begin{array}{cc}
e^{ik\theta }, & k=0,1,\dots ,2^{n} \\ 
e^{i2^{n}\theta }, & k=2^{n}+1,2^{n}+2,\dots%
\end{array}%
\right. \text{,}  \label{Definition of Ytheta} \\
\text{then }\left\Vert Y_{\theta }^{n}\right\Vert _{2-var}^{2}+\left\Vert
A\left( Y_{\theta }^{n}\right) \right\Vert _{1-var}\leq 61\times 2^{n-1}\pi
\theta \text{, \ }\forall n\geq \max \left\{ \log _{2}\left( \frac{2\pi }{%
\theta }\right) ,1\right\} \text{.}  \notag
\end{gather}
\end{lemma}

\begin{proof}
For fixed $\theta \in \left( 0,2\pi \right) $, we do analysis for fixed $%
n\geq \max \left\{ \log _{2}\left( \frac{2\pi }{\theta }\right) ,1\right\} $%
. In the following, we do not specify the dependence on $\theta $ or on $n$,
and $Y$ denotes $Y_{\theta }^{n}$ in the statement.

Define continuous path $\widetilde{Y}:[0,\infty )\rightarrow 
\mathbb{C}
$ as 
\begin{equation*}
\widetilde{Y}\left( t\right) =\left\{ 
\begin{array}{cc}
e^{it\theta }, & \text{ }t\in \left[ 0,2^{n}\right] \\ 
e^{i2^{n}\theta }, & \text{ }t\in \left( 2^{n},\infty \right)%
\end{array}%
\right. .
\end{equation*}%
Denote $N:=\left[ \frac{2^{n}\theta }{2\pi }\right] $ (the integer part of $%
\frac{2^{n}\theta }{2\pi }$), so $N\geq 1$ since $n\geq \log _{2}(\frac{2\pi 
}{\theta })$. Denote $t_{j}:=\frac{2j\pi }{\theta }$ for $j=0,1,\dots ,N$.

First, we estimate $2$-variation of $Y$. Then since $\widetilde{Y}\left(
t_{j}\right) =1$, $j=0,1,\dots ,N$, similar as $\left( \ref{Bound of 2-var
by 2-var of two processes}\right) $ in Lemma \ref{Lemma: key inequality for
path and area} on p\pageref{Bound of 2-var by 2-var of two processes}, we
have ($||\widetilde{Y}||_{2-var,\left[ t_{j},t_{j+1}\right] }^{2}\leq 4\pi
^{2}$, $j=0,1,\dots ,N-1$) 
\begin{eqnarray}
&&||\widetilde{Y}||_{2-var}^{2}  \label{inner estimation of Ytilde on [0,2n]}
\\
&\leq &3\left( \sum_{j=0}^{N-1}||\widetilde{Y}||_{2-var,\left[ t_{j},t_{j+1}%
\right] }^{2}+||\widetilde{Y}||_{2-var,\left[ t_{N},2^{n}\right]
}^{2}+||e^{i2^{n}\theta }-1||^{2}\right)  \notag \\
&\leq &3\left( \sum_{j=0}^{N-1}||\widetilde{Y}||_{2-var,\left[ t_{j},t_{j+1}%
\right] }^{2}+4\pi ^{2}\right) \leq 12\pi ^{2}\left( N+1\right) \leq 24\pi
^{2}N.  \notag
\end{eqnarray}%
Thus, since $Y$ is a discretization of $\widetilde{Y}$, we have%
\begin{equation*}
\left\Vert Y\right\Vert _{2-var}^{2}\leq ||\widetilde{Y}||_{2-var}^{2}\leq
24\pi ^{2}N.
\end{equation*}

For $A\left( Y\right) $, denote%
\begin{equation}
n_{j}:=[t_{j}]\text{ (the integer part of }t_{j}\text{).}
\label{inner definition of n_j}
\end{equation}%
Since $t_{j+1}-t_{j}=\frac{2\pi }{\theta }>1$, we have $n_{j_{1}}\neq
n_{j_{2}}$ when $j_{1}\neq j_{2}$, and $n_{j}+1\leq n_{j+1}$, $j=0,1,\dots
,N-1$. Denote%
\begin{equation}
Y^{1}\left( t\right) :=\left\{ 
\begin{array}{cc}
\widetilde{Y}\left( t\right) , & t=0,1,\dots ,2^{n}\text{ or }%
t_{0},t_{1},\dots ,t_{N} \\ 
e^{i2^{n}\theta }, & t=2^{n}+1,2^{n}+2\dots%
\end{array}%
\right. \text{,}  \label{inner definition of Y1}
\end{equation}%
and%
\begin{equation}
Y^{2}\left( t\right) :=\left\{ 
\begin{array}{cc}
1, & t=t_{0},t_{1},\dots ,t_{N} \\ 
e^{i2^{n}\theta }, & t=2^{n},2^{n}+1,2^{n}+2\dots%
\end{array}%
\right. \text{.}  \label{inner definition of Y2}
\end{equation}%
Since $Y^{1}$ is obtained by inserting $t_{j}$ between $n_{j}$ and $n_{j}+1$
in $Y$, based on Lemma \ref{Lemma relation between area of a path and its
discretization} on p\pageref{Lemma relation between area of a path and its
discretization}, we have, for any $k_{1}<k_{2}$, 
\begin{equation*}
A\left( Y\right) \left( k_{1},k_{2}\right) =A\left( Y^{1}\right) \left(
k_{1},k_{2}\right) -\sum_{j,\left[ n_{j},n_{j}+1\right] \subseteq \left[
k_{1},k_{2}\right] }A\left( Y^{1}\right) \left( n_{j},n_{j}+1\right) \text{.}
\end{equation*}%
Thus, 
\begin{equation}
\left\Vert A\left( Y\right) \right\Vert _{1-var}\leq \left\Vert A\left(
Y^{1}\right) \right\Vert _{1-var}+\sum_{j=0}^{N}\left\Vert A\left(
Y^{1}\right) \left( n_{j},n_{j}+1\right) \right\Vert \text{.}
\label{inner area of Y 1}
\end{equation}%
Since $t_{j}$ is the only point between $n_{j}$ and $n_{j}+1$ in $Y^{1}$,
based on the definition of area (on p\pageref{Definition of area}), we have%
\begin{gather}
\sum_{j=0}^{N}\left\Vert A\left( Y^{1}\right) \left( n_{j},n_{j}+1\right)
\right\Vert \leq \sum_{j=0}^{N}\left\Vert Y^{1}\left( n_{j},t_{j}\right)
\right\Vert \left\Vert Y^{1}\left( t_{j},n_{j}+1\right) \right\Vert
\label{inner area of Y 2} \\
\leq \frac{1}{2}\sum_{j=0}^{N}\left( \left\Vert Y^{1}\left(
n_{j},t_{j}\right) \right\Vert ^{2}+\left\Vert Y^{1}\left(
t_{j},n_{j}+1\right) \right\Vert ^{2}\right) \leq \frac{1}{2}\left\Vert
Y^{1}\right\Vert _{2-var}^{2}\text{.}  \notag
\end{gather}%
Thus, combine $\left( \ref{inner area of Y 1}\right) $ and $\left( \ref%
{inner area of Y 2}\right) $, we have%
\begin{equation}
\left\Vert A\left( Y\right) \right\Vert _{1-var}\leq \left\Vert A\left(
Y^{1}\right) \right\Vert _{1-var}+\frac{1}{2}\left\Vert Y^{1}\right\Vert
_{2-var}^{2}\text{.}  \label{inner area of Y 3}
\end{equation}%
Based on definition of $Y^{1}$ and $Y^{2}$ at $\left( \ref{inner definition
of Y1}\right) $ and $\left( \ref{inner definition of Y2}\right) $, $Y^{2}$
is a subsequence of $Y^{1}$ (since $Y^{1}\left( t_{j}\right) =1$), and $%
\left\Vert A\left( Y^{2}\right) \right\Vert _{1-var}=0$ (because $Y^{2}$
have only two possible values). According to $\left( \ref{Bound for 1var of
area}\right) $ (on p\pageref{Lemma: key inequality for path and area}), (use 
$\left\Vert A\left( Y^{1}\right) \right\Vert _{1-var,[t_{N},2^{n}]}=0$) 
\begin{eqnarray*}
\left\Vert A\left( Y^{1}\right) \right\Vert _{1-var} &\leq
&||Y^{1}||_{2-var}^{2}+\sum_{j=0}^{N-1}\left\Vert A\left( Y^{1}\right)
\right\Vert _{1-var,[t_{j},t_{j+1}]}+\left\Vert A\left( Y^{2}\right)
\right\Vert _{1-var} \\
&=&||Y^{1}||_{2-var}^{2}+\sum_{j=0}^{N-1}\left\Vert A\left( Y^{1}\right)
\right\Vert _{1-var,[t_{j},t_{j+1}]}\text{,}
\end{eqnarray*}%
Combined with $\left( \ref{inner area of Y 3}\right) $, we get 
\begin{equation}
\left\Vert A\left( Y\right) \right\Vert _{1-var}\leq \frac{3}{2}%
||Y^{1}||_{2-var}^{2}+\sum_{j=0}^{N-1}\left\Vert A\left( Y^{1}\right)
\right\Vert _{1-var,[t_{j},t_{j+1}]}\text{.}
\label{inner conclusion of estimation of area}
\end{equation}

Then we estimate the two components in $\left( \ref{inner conclusion of
estimation of area}\right) $. \ For $||Y^{1}||_{2-var}^{2}$, based on $%
\left( \ref{inner estimation of Ytilde on [0,2n]}\right) $, we have 
\begin{equation}
||Y^{1}||_{2-var}^{2}\leq ||\widetilde{Y}||_{2-var}^{2}\leq 24\pi ^{2}N\text{%
.}  \label{inner area 1}
\end{equation}%
For $\left\Vert A\left( Y^{1}\right) \right\Vert _{1-var,[t_{j},t_{j+1}]}$,
we have the estimate that%
\begin{equation}
\sum_{j=0}^{N-1}\left\Vert A\left( Y^{1}\right) \right\Vert
_{1-var,[t_{j},t_{j+1}]}\leq \pi ^{2}N\text{.}  \label{inner area 2}
\end{equation}%
Actually, $Y^{1}$ on $\left[ t_{j},t_{j+1}\right] $ describes a simple
convex polygon, with unit circle its circumcircle, so $1$-variation of $%
A\left( Y^{1}\right) $ is bounded by $\pi ^{2}$.

Therefore, combine $\left( \ref{inner conclusion of estimation of area}%
\right) $, $\left( \ref{inner area 1}\right) $ and $\left( \ref{inner area 2}%
\right) $, we get%
\begin{equation*}
\left\Vert A\left( Y\right) \right\Vert _{1-var}\leq 37\pi ^{2}N\text{,}
\end{equation*}%
and ($N\leq \frac{2^{n}\theta }{2\pi }$)%
\begin{equation*}
\left\Vert Y\right\Vert _{2-var}^{2}+\left\Vert A\left( Y\right) \right\Vert
_{1-var}\leq 61\pi ^{2}N\leq 61\times 2^{n-1}\pi \theta \text{.}
\end{equation*}
\end{proof}

\medskip

\begin{example}
\label{Example of L2 Fourier series finite 2-variation a.e. but dont satisfy
regularity condition}There exists an $L^{2}$ Fourier series $%
\sum_{n=1}^{\infty }c_{n}e^{in\theta }$, s.t. its partial sum process is a
geometric $2$-rough process, but $\sum_{n}\log _{2}\left( n+1\right)
\left\vert c_{n}\right\vert ^{2}=\infty $.
\end{example}

\medskip

The same example can be modified to any Weyl multiplier growing strictly
faster than $\{\left( \log _{2}\log _{2}n\right) ^{2}\}$, as in Example \ref%
{Example lnn is not sufficient for Fourier series} proved after this example.

\smallskip

\begin{proof}
Define%
\begin{equation}
f\left( \theta \right) =\sum_{n=1}^{\infty }\frac{1}{n2^{\frac{n}{2}}}%
\sum_{k=2^{n}+1}^{2^{n+1}}e^{ik\theta }\text{, \ }\theta \in \lbrack 0,2\pi
).  \label{inner definition of example finite a.e. 2var}
\end{equation}%
Then $\left\vert c_{n}\right\vert ^{2}\sim n^{-1}\left( \log _{2}n\right)
^{-2}$, so $f$ is in $L^{2}$ and $\sum_{n}\log _{2}\left( n+1\right)
\left\vert c_{n}\right\vert ^{2}=\infty $. Denote $X$ as the partial sum
process of $f$, then (when $\theta \neq 0$)%
\begin{equation}
X_{\theta }\left( k\right) =\frac{e^{i\left( 2^{n}+1\right) \theta }}{n2^{%
\frac{n}{2}}\left( 1-e^{i\theta }\right) }\left( 1-e^{i\left( k-2^{n}\right)
\theta }\right) +X_{\theta }\left( 2^{n}\right) \text{, }k=2^{n}+1,\dots
,2^{n+1}.  \label{inner example expression of partial sum process}
\end{equation}%
Define $X^{1}$ as $X^{1}\left( n\right) :=X\left( 2^{n}\right) $, $\forall
n\in 
\mathbb{N}
$. Then $X^{1}$ can be enhanced into a geometric $2$-rough process (if
denote $v_{n}=2^{-\frac{n}{2}}\sum_{k=2^{n}+1}^{2^{n+1}}e^{ik\theta }$, then 
$X^{1}$ is the partial sum process of $\sum_{n}n^{-1}v_{n}$, and use Theorem %
\ref{Theorem partial sum process of general orthogonal expansion as
geometric rough process}). Based on Lemma \ref{Lemma: key inequality for
path and area} on p\pageref{Lemma: key inequality for path and area}, we are
done if we can prove, 
\begin{equation}
\sum_{n=0}^{\infty }\left( \left\Vert X_{\theta }\right\Vert _{2-var,\left[
2^{n},2^{n+1}\right] }^{2}+\left\Vert A\left( X_{\theta }\right) \right\Vert
_{1-var,\left[ 2^{n},2^{n+1}\right] }\right) <\infty \text{ a.e..}
\label{inner what to prove in example}
\end{equation}%
When $\theta =0$, $\left\Vert X_{\theta }\right\Vert _{1-var}=\infty $, so $%
\left( X_{0},A_{0}\right) $ is not a geometric $2$-rough path. We prove that 
$\left( \ref{inner what to prove in example}\right) $ holds for any $\theta
\in \left( 0,2\pi \right) $.

In $\left( \ref{inner example expression of partial sum process}\right) $,
since $\frac{e^{i\left( 2^{n}+1\right) \theta }}{n2^{\frac{n}{2}}\left(
1-e^{i\theta }\right) }$ and $X_{\theta }\left( 2^{n}\right) $ are constants
for fixed $\theta $ and $n$, using Lemma \ref{Lemma estimation for fixed
theta}, we have, for any $n\geq \max \left\{ \log _{2}\left( \frac{2\pi }{%
\theta }\right) ,1\right\} $, (with $Y_{\theta }^{n}$ defined at $\left( \ref%
{Definition of Ytheta}\right) $) 
\begin{eqnarray}
&&\left\Vert X_{\theta }\right\Vert _{2-var,\left[ 2^{n},2^{n+1}\right]
}^{2}+\left\Vert A_{\theta }\right\Vert _{1-var,\left[ 2^{n},2^{n+1}\right] }
\label{inner convergent factor} \\
&=&\frac{1}{4n^{2}2^{n}\sin ^{2}\frac{\theta }{2}}\left( \left\Vert
Y_{\theta }^{n}\right\Vert _{2-var}^{2}\text{ }+\left\Vert A\left( Y_{\theta
}^{n}\right) \right\Vert _{1-var}\right) \leq \frac{61\pi \theta }{8\sin ^{2}%
\frac{\theta }{2}}\frac{1}{n^{2}}\text{.}  \notag
\end{eqnarray}
\end{proof}

\medskip

Although in the example above, $\left( X_{\theta },A_{\theta }\right) $ is
of finite $2$-rough norm when $\theta \neq 0$, the integration $\int_{\Omega
}\left\Vert \mathbf{X}_{\theta }\right\Vert _{G^{\left( 2\right) }}d\theta $
is not finite, and the problem occurs at $0$ or $2\pi $, as one may see.
After some modifications, we can push the result a little bit further. The
convergent factor $n^{-2}$ only appeared in $\left( \ref{inner convergent
factor}\right) $, so one could modify the example to%
\begin{equation}
\sum_{n=1}^{\infty }\frac{1}{a^{\frac{1}{2}}\left( n\right) 2^{\frac{n}{2}}}%
\sum_{k=2^{n}+1}^{2^{n+1}}e^{ik\theta },
\label{generalised example finite 2var}
\end{equation}%
for any positive $\left\{ a\left( n\right) \right\} $ satisfying $\sum \frac{%
1}{a\left( n\right) }<\infty $. However, the long time behavior will then
cause a problem. Denote $X^{1}$ as $X^{1}\left( n\right) :=X\left(
2^{n}\right) $, $n\in 
\mathbb{N}
$. According to Theorem \ref{Theorem partial sum process of general
orthogonal expansion as geometric rough process}, we know that if $%
\sum_{n}\left( \log _{2}n\right) ^{2}/a\left( n\right) <\infty $, then $%
X^{1} $ is a geometric $2$-rough process, so it will not be a problem under
that condition, while the local regularity is controlled by $\left( \ref%
{inner convergent factor}\right) $. In that case, based on Lemma \ref{Lemma:
key inequality for path and area}(on p\pageref{Lemma: key inequality for
path and area}), the partial sum process of $\left( \ref{generalised example
finite 2var}\right) $ is a geometric $2$-rough process. Therefore, we can
generalize Example \ref{Example of L2 Fourier series finite 2-variation a.e.
but dont satisfy regularity condition}:

\bigskip

\noindent%
\textbf{Example} \textbf{\ref{Example lnn is not sufficient for Fourier
series}} \ \textit{Suppose }$\left\{ w\left( n\right) \right\} $\textit{\ is
a Weyl multiplier that }$n\mapsto \frac{w\left( n\right) }{\left( \log
_{2}\log _{2}n\right) ^{2}}\,$\textit{\ is strictly increasing from some
point on and }$\lim_{n\rightarrow \infty }\frac{w\left( n\right) }{\left(
\log _{2}\log _{2}n\right) ^{2}}=\infty $\textit{. Then there exists an\ }$%
L^{2}$ \textit{Fourier series }$\sum_{n=1}^{\infty }c_{n}e^{in\theta }$,%
\textit{\ such that its partial sum process is a geometric }$2$\textit{%
-rough process, but }$\sum_{n}w\left( n\right) \left\vert c_{n}\right\vert
^{2}=\infty $\textit{.}

\bigskip

\begin{proof}
In light of Example \ref{Example of L2 Fourier series finite 2-variation
a.e. but dont satisfy regularity condition}, we only have to prove the
statement for $\left\{ w\left( n\right) \right\} $ growing slower than $%
\left\{ \log _{2}\left( n+1\right) \right\} $. Thus, assume $%
\lim_{n\rightarrow \infty }\frac{w\left( 2^{n+1}\right) }{w\left(
2^{n}\right) }=1$. According to the condition of this example, assume $N\geq
2$ is such an integer, that $n\mapsto \frac{w\left( 2^{n}\right) }{\left(
\log _{2}n\right) ^{2}}$ is strictly increasing for all $n\geq N$. Let $%
r:[N-1,\infty )\rightarrow 
\mathbb{R}
^{+}$ be a differentiable path satisfying $r^{\prime }\left( t\right) \geq 0$
for all $t\geq N-1$, and 
\begin{equation}
r\left( n\right) =\frac{w\left( 2^{n}\right) }{\left( \log _{2}n\right) ^{2}}%
\ \text{, }n\geq N\text{, with }r\left( N-1\right) =\frac{1}{2}r\left(
N\right) \text{.}  \label{inner setting value of r at integers}
\end{equation}%
Moreover, we assume, 
\begin{equation}
r^{\prime }\left( n\right) =\frac{r\left( n+1\right) -r\left( n-1\right) }{2}%
\text{ , }n\geq N\text{, with }r_{+}^{\prime }\left( N-1\right) =\frac{1}{2}%
r^{\prime }\left( N\right) .
\label{inner setting value of derivative of r at integers}
\end{equation}%
Such kind of function $r$ exists: The problem boils down to, for fixed real
numbers $k>0$, $k_{1}>0$, $k_{2}>0$, constructing a one dimensional
non-decreasing differentiable function $f$, defined on $\left[ 0,1\right] $,
satisfying $f\left( 0\right) =0$, $f\left( 1\right) =k$, $f_{+}^{\prime
}\left( 0\right) =k_{0}$ and $f_{-}^{\prime }\left( 1\right) =k_{1}$. Then $%
f $ exists, if there exists a continuous function $\rho $, defined on $\left[
0,1\right] $, satisfying $\rho \left( t\right) \geq 0$, $\rho \left(
0\right) =k_{0}$, $\rho \left( 1\right) =k_{1}$, $\int_{0}^{1}\rho \left(
t\right) dt=k$. Such $\rho $ clearly exists, so $f\left( t\right)
=\int_{0}^{t}\rho \left( s\right) ds$ satisfies the condition of $f$. Thus,
we can construct $r$ by first setting its value at integers by $\left( \ref%
{inner setting value of r at integers}\right) $ and $\left( \ref{inner
setting value of derivative of r at integers}\right) $, then on $\left[ n,n+1%
\right] $ for integer $n\geq N-1$ use the construction of $f$ as above. In
this way, $r$ is absolutely continuous on any finite interval $\left[ a,b%
\right] \subseteq \lbrack N-1,\infty )$ (its derivative is continuous, so $r$
is Lipschitz on any finite interval), thus we have $\int_{a}^{b}r^{\prime
}\left( t\right) dt=r\left( b\right) -r\left( a\right) $. As an application,
use $\left( \ref{inner setting value of derivative of r at integers}\right) $%
, 
\begin{equation}
r^{\prime }\left( n\right) =\frac{1}{2}\int_{n-1}^{n+1}r^{\prime }\left(
t\right) dt.  \label{inner relation r(n)}
\end{equation}%
\begin{equation*}
\text{Let }\frac{1}{a\left( n\right) }=\frac{r^{\prime }\left( n\right) }{%
r\left( n\right) \sqrt{\left( \log _{2}n\right) ^{2}w\left( 2^{n}\right) }}%
\text{; define }f\left( \theta \right) :=\sum_{n=N}^{\infty }\frac{1}{a^{%
\frac{1}{2}}\left( n\right) 2^{\frac{n}{2}}}\sum_{k=2^{n}+1}^{2^{n+1}}e^{ik%
\theta }\text{.}
\end{equation*}%
Since $\lim_{n\rightarrow \infty }\frac{w\left( 2^{n+1}\right) }{w\left(
2^{n}\right) }=1$, we have $\lim_{n\rightarrow \infty }\frac{r\left(
n+1\right) }{r\left( n\right) }=1$, and using $\left( \ref{inner relation
r(n)}\right) $, we get%
\begin{eqnarray*}
\sum_{n\geq N}\frac{\left( \log _{2}n\right) ^{2}}{a\left( n\right) }
&=&\sum_{n\geq N}\frac{\left( \log _{2}n\right) ^{2}r^{\prime }\left(
n\right) }{r\left( n\right) \sqrt{\left( \log _{2}n\right) ^{2}w\left(
2^{n}\right) }}=\sum_{n\geq N}\frac{r^{\prime }\left( n\right) }{\left(
r\left( n\right) \right) ^{\frac{3}{2}}}\sim \sum_{n\geq N}\frac{r^{\prime
}\left( n\right) }{\left( r\left( n+1\right) \right) ^{\frac{3}{2}}} \\
&\leq &\lim_{M\rightarrow \infty }\sum_{n=N}^{M}\frac{1}{2}\int_{n-1}^{n+1}%
\frac{r^{\prime }}{r^{\frac{3}{2}}}dt\leq \lim_{M\rightarrow \infty
}\int_{N-1}^{M+1}\frac{dr}{r^{\frac{3}{2}}}=\frac{2}{\sqrt{r\left(
N-1\right) }}<\infty .
\end{eqnarray*}%
Thus, by following exactly the same reasoning of Example \ref{Example of L2
Fourier series finite 2-variation a.e. but dont satisfy regularity condition}%
, the partial sum process of $f$ is a geometric $2$-rough process. On the
other hand, since $\left\{ w\left( n\right) \right\} $ is non-decreasing, so%
\begin{eqnarray*}
&&\sum_{n\geq 2^{N}+1}w\left( n\right) \left\vert c_{n}\right\vert ^{2}\geq
\sum_{n\geq N}\left( \sum_{k=2^{n}+1}^{2^{n+1}}\left\vert c_{k}\right\vert
^{2}\right) w\left( 2^{n}\right) =\sum_{n\geq N}\frac{w\left( 2^{n}\right) }{%
a\left( n\right) } \\
&=&\sum_{n\geq N}\frac{r^{\prime }\left( n\right) }{\sqrt{r\left( n\right) }}%
\overset{\left( \ref{inner relation r(n)}\right) }{\geq }\lim_{M\rightarrow
\infty }\sum_{n=N}^{M}\frac{1}{2}\int_{n}^{n+1}\frac{r^{\prime }}{\sqrt{r}}%
dt=\lim_{M\rightarrow \infty }\frac{1}{2}\int_{N}^{M+1}\frac{dr}{\sqrt{r}}%
=\infty .
\end{eqnarray*}
\end{proof}

\section{Example of an $L^{2}$ Fourier series with infinite $2$-variation
almost everywhere}

Before construction, we prove the upper semi-continuity of the cumulative
distribution function of $p$-variation.

\begin{lemma}
\label{Lemma for the increase of p-variation}Suppose $\left\{ X_{n}\right\}
_{n=1}^{\infty }$ and $X$ are continuous processes, defined on probability
space $\left( \Omega ,\mathcal{F},P\right) $, taking value in $%
\mathbb{R}
^{d}$, and $X_{n}$ converge to $X$ in distribution as $n$ tends to infinity.
Then for any $p\geq 1$, $C\geq 0$, 
\begin{equation*}
\overline{\lim }_{n\rightarrow \infty }P\left( \left\Vert X_{n}\right\Vert
_{p-var}\leq C\right) \leq P\left( \left\Vert X\right\Vert _{p-var}\leq
C\right) .
\end{equation*}
\end{lemma}

\begin{proof}
$C[0,\infty )$, the space of continuous $%
\mathbb{R}
^{d}$-valued functions on $[0,\infty )$, is a complete, separable metric
space when equipped with the metric:%
\begin{equation*}
\rho \left( \omega _{1},\omega _{2}\right) :=\sum_{n=1}^{\infty }\frac{1}{%
2^{n}}\max_{0\leq t\leq n}\left( \left\vert \omega _{1}\left( t\right)
-\omega _{2}\left( t\right) \right\vert \wedge 1\right) .
\end{equation*}%
$X_{n}$ and $X$ are random variables taking values in $\left( C[0,\infty ),%
\mathcal{B}\left( C[0,\infty )\right) \right) $. According to Skorohod's
theorem, there exists $\widetilde{X_{n}}$ and $\widetilde{X}$ on an
auxiliary space, s.t. $X_{n}\overset{D}{=}\widetilde{X_{n}}$ , $X\overset{D}{%
=}\widetilde{X}$ , and $\widetilde{X_{n}}$ converges to $\widetilde{X}$
a.e.. Use Fatou's lemma and lower semi-continuity of $p$-variation,%
\begin{eqnarray*}
&&\underline{\lim }_{n\rightarrow \infty }P\left( \left\Vert
X_{n}\right\Vert _{p-var}>C\right) =\underline{\lim }_{n\rightarrow \infty
}P\left( \left\Vert \widetilde{X_{n}}\right\Vert _{p-var}>C\right) \\
&\geq &P\left( \underline{\lim }_{n\rightarrow \infty }\left\{ \left\Vert 
\widetilde{X_{n}}\right\Vert _{p-var}>C\right\} \right) =P\left( \underline{%
\lim }_{n\rightarrow \infty }\left\Vert \widetilde{X_{n}}\right\Vert
_{p-var}>C\right) \\
&\geq &P\left( \left\Vert \widetilde{X}\right\Vert _{p-var}>C\right)
=P\left( \left\Vert X\right\Vert _{p-var}>C\right) .
\end{eqnarray*}
\end{proof}

As a trivial Corollary, for any $\alpha >0$, $p\geq 1$,%
\begin{equation}
\underline{\lim }_{n\rightarrow \infty }E\left( \left\Vert X_{n}\right\Vert
_{p-var}^{\alpha }\right) \geq E\left( \left\Vert X\right\Vert
_{p-var}^{\alpha }\right)
\label{inner lower semi continuity of expectation of p-variation}
\end{equation}

\begin{corollary}
\label{Corollary for the increase of p-variation of random walk}Suppose $%
S_{k}$ is the sum of first $k$ terms of a sequence of i.i.d. random
variables with mean $0$ and variance $1$. Define $\xi _{n}\,\ $as the
process on $\left[ 0,1\right] $ obtained by interpolating $S_{k}/n^{\frac{1}{%
2}}$ at $k/n$, $k=0,1,\dots ,n$. Then for any $C\geq 0$,%
\begin{equation*}
\lim_{n\rightarrow \infty }P\left( \left\Vert \xi _{n}\right\Vert
_{2-var}>C\right) =1.
\end{equation*}
\end{corollary}

\begin{proof}
$\xi _{n}$ converge in distribution to the Wiener process $W$, use Lemma \ref%
{Lemma for the increase of p-variation} and that Wiener process is of
infinite $2$-variation a.e., we get 
\begin{equation*}
\underline{\lim }_{n\rightarrow \infty }P\left( \left\Vert \xi
_{n}\right\Vert _{2-var,\left[ 0,1\right] }>C\right) \geq P\left( \left\Vert
W\right\Vert _{2-var,\left[ 0,1\right] }>C\right) =1.
\end{equation*}
\end{proof}

In fact, it is proved in \cite{J. Qian} (with non-trivial reasoning) that
there exists constant $c>0$ such that, if assume the i.i.d. random variables
have finite $2+\delta $ moment for some $\delta >0$, then $%
\lim_{n\rightarrow \infty }P\left( \left\Vert \xi _{n}\right\Vert
_{2-var}^{2}\geq c\ln \ln n\right) =1$.

If we were working with Rademacher functions ($r_{k}\left( t\right) =sgn\sin
\left( 2^{k}\pi t\right) $, $t\in \left[ 0,1\right] $, $k\geq 1$), the
construction would be clearer, because $r_{k}$ are independent. Glue pieces
of rescaled random walks together, where each piece provides sufficiently
large $2$-variation, then a.e. infinite $2$-variation follows from
Borel-Cantelli lemma. It is similar for Fourier series, only that we pick
out those trigonometric functions which resemble an i.i.d. sequence. (For
any $m$ and $n$, $e^{2\pi in\theta }$ and $e^{2\pi im\theta }$ are never
independent: suppose $\theta $ is uniformly distributed on $\left[ 0,1\right]
$, with a binary expansion $\sum_{k=1}^{\infty }\theta _{k}2^{-k}$, then
both $\left\{ n\theta \right\} $ and $\left\{ m\theta \right\} $ -- the
fractional part of $n\theta $ and $m\theta $ -- depend on $\sigma (\left\{
\theta _{k}\right\} _{k\geq K})$ for some $K\geq 1$, comparing to Rademacher
system, which is independent because $r_{k}=-2\theta _{k}+1$.) However, (we
suppose that) there are far more trigonometric sequences, which do not
exhibit random behavior, but with a heavy $L^{2}$ tail and infinite $2$%
-variation almost everywhere.

Suppose we have a sequence of integers$\overset{m_{1}}{\text{ }\overbrace{%
n_{1},n_{1},\dots ,n_{1}}},\dots ,\overset{m_{k}}{\overbrace{%
n_{k},n_{k},\dots ,n_{k}}},\dots $ where $n_{k}$, $m_{k}$, $k\geq 1$ are
integers. Denote the partial sum of this sequence as $s_{0}=0$, $%
s_{k}=\sum_{j=1}^{k}m_{j}n_{j}$. Suppose $\theta $ is uniformly distributed
on $\left[ 0,1\right] $, and $\theta _{k}$ is the $k$th digit of the binary
expansion of $\theta $, i.e. $\theta =\sum_{k=1}^{\infty }\theta _{k}2^{-k}$%
. One can check that $\left\{ \theta _{k}\right\} _{k\geq 1}$ are i.i.d.
random variables satisfying $P\left( \theta _{k}=1\right) =P\left( \theta
_{k}=0\right) =\frac{1}{2}$.

\begin{definition}
Define a sequence of random variables%
\begin{equation}
\varsigma _{i}^{\left( n_{k}\right) }=\cos \left( 2\pi \sum_{j=1}^{n_{k}}%
\frac{\theta _{s_{k-1}+\left( i-1\right) n_{k}+j}}{2^{j}}\right) \text{, }%
i=1,2,\dots ,m_{k}\text{, }k\geq 1,  \label{Definition of random variables}
\end{equation}%
where $m_{k}$, $n_{k}$, $s_{k}$, and $\theta _{k}$ are defined above.
\end{definition}

$\{\varsigma _{i}^{\left( n_{k}\right) },1\leq i\leq m_{k},k\geq 1\}$ are
independent with mean $0$ variance $\frac{1}{2}$, and for each fixed $k$, $%
\{\varsigma _{i}^{\left( n_{k}\right) },1\leq i\leq m_{k}\}$ are identically
distributed. Moreover,%
\begin{equation}
\left\vert \varsigma _{i}^{\left( n_{k}\right) }-\cos \left( 2\pi
2^{s_{k-1}+\left( i-1\right) n_{k}}\theta \right) \right\vert \leq \frac{\pi 
}{2^{n_{k}-1}}.  \label{Bound of error in uniform norm}
\end{equation}%
Suppose $X$ and $Y$ are respectively the partial sum process of%
\begin{equation*}
\text{ }f\left( \theta \right) =\sum_{k=1}^{\infty }\frac{1}{k\sqrt{m_{k}}}%
\sum_{j=1}^{m_{k}}\cos \left( 2\pi 2^{s_{k-1}+\left( j-1\right) n_{k}}\theta
\right) \text{ and }\varsigma =\sum_{k=1}^{\infty }\frac{1}{k\sqrt{m_{k}}}%
\sum_{j=1}^{m_{k}}\varsigma _{j}^{\left( n_{k}\right) }.
\end{equation*}%
Then by showing that $Y$ is of infinite $2$-variation a.e., and choosing $%
n_{k}$ and $m_{k}$ to control the cumulated error produced by $\left( \ref%
{Bound of error in uniform norm}\right) $, we can prove that $X$ of infinite 
$2$-variation a.e.. However, the estimation in Example \ref{Example of
almost everywhere infinite quadratic variation} (re-stated below) forces us
to choose $m_{k}$ before $n_{k}$. Therefore, we need a result of uniform
growth of $2$-variation of random walks produced by $\varsigma _{i}^{\left(
n_{k}\right) }$ for different $k$s.

\begin{definition}
\label{Definition of Ymn}Define $Y_{m}^{n}$ as the process on $\left[ 0,1%
\right] $ by interpolating $\sum_{i=1}^{k}\varsigma _{i}^{\left( n\right)
}/m^{\frac{1}{2}}$ at $k/m$, $k=1,2,\dots ,m$, where $\varsigma _{i}^{\left(
n\right) }$, $i=1,2,\dots ,m$, are as defined in $\left( \ref{Definition of
random variables}\right) $.
\end{definition}

\begin{lemma}
\label{Lemma of uniform increase}For any constant $C\geq 0$,%
\begin{equation*}
\underline{\lim }_{m\rightarrow \infty }\underline{\lim }_{n\rightarrow
\infty }P\left( \left\Vert Y_{m}^{n}\right\Vert _{2-var}>C\right) =1.
\end{equation*}
\end{lemma}

\begin{proof}
Suppose $\left\{ \theta _{i}\right\} _{i=1}^{m}$ are independent random
variables uniformly distributed on $\left[ 0,1\right] $, and $Y_{m}$ the
continuous process got by interpolating $(\sum_{i=1}^{k}\cos \theta _{i})/m^{%
\frac{1}{2}}$ at $k/m$. Since $P(\varsigma _{i}^{\left( n\right) }=\cos
(2\pi k2^{-n}))=2^{-n}$, $k=0,\dots ,2^{n}-1$, so $\varsigma _{i}^{\left(
n\right) }$ converge to $\cos \theta _{i}$ in distribution as $n\rightarrow
\infty $. Noting that $m$ is fixed, and $\varsigma _{i}^{\left( n\right) }$, 
$i=1,2,\dots m$, are independent, so $Y_{m}^{n}$ converge to $Y_{m}$ in
distribution as $n\rightarrow \infty $. Use Lemma \ref{Lemma for the
increase of p-variation} and Corollary \ref{Corollary for the increase of
p-variation of random walk}, 
\begin{multline}
\underline{\lim }_{m\rightarrow \infty }\underline{\lim }_{n\rightarrow
\infty }P\left( \left\Vert Y_{m}^{n}\right\Vert _{2-var,\left[ 0,1\right]
}>C\right)  \notag \\
\geq \underline{\lim }_{m\rightarrow \infty }P\left( \left\Vert
Y_{m}\right\Vert _{2-var,\left[ 0,1\right] }>C\right) =1.
\end{multline}
\end{proof}

Now, we are prepared to construct our series.

\bigskip

\noindent%
\textbf{Example \ \ref{Example of almost everywhere infinite quadratic
variation}} \ \textit{There exists an }$L^{2}$\textit{\ Fourier series whose
partial sum process has infinite }$2$\textit{-variation almost everywhere.}

\bigskip

\begin{proof}
According to Lemma \ref{Lemma of uniform increase}, there exists a sequence
of integers, $\left\{ M_{s}\right\} _{s\geq 2}$, s.t. $\forall m\geq M_{s}$, 
$\exists N\left( s,m\right) $, s.t. $\forall n\geq N\left( s,m\right) $, 
\begin{equation*}
P\left( \left\Vert Y_{m}^{n}\right\Vert _{2-var}^{2}>s^{2}\right) \geq \frac{%
1}{s}.
\end{equation*}%
Set $m_{k}:=\max_{1\leq s\leq k}M_{s}$. Choose $\left\{ n_{k}\right\}
_{k=1}^{\infty }$, s.t. $n_{k}\geq N\left( k,m_{k}\right) $, $2^{n_{k}}>k%
\sqrt{m_{k}}$, and $n_{k+1}>n_{k}$. Hence,%
\begin{equation*}
P\left( \left\Vert Y_{m_{k}}^{n_{k}}\right\Vert _{2-var}^{2}>k^{2}\right)
\geq \frac{1}{k}\text{, and }\sum_{k=2}^{\infty }\frac{\sqrt{m_{k}}}{%
k2^{n_{k}}}<\infty .
\end{equation*}%
Denote $Y$ as the continuous process constructed on $[0,\infty )$ by
patching up $Y_{m_{k}}^{n_{k}}/k$, $k\geq 2$. Then based on the elementary
inequality: $a^{2}\geq b^{2}/2-\left( a-b\right) ^{2}$, we have: ($X$ is the
partial sum process of corresponding Fourier series)%
\begin{equation*}
\left\Vert X\right\Vert _{2-var}^{2}\geq \frac{1}{2}\left\Vert Y\right\Vert
_{2-var}^{2}-\left( 2\pi \sum_{k=2}^{\infty }\frac{\sqrt{m_{k}}}{k2^{n_{k}}}%
\right) ^{2}\geq \frac{1}{2}\left\Vert Y\right\Vert _{2-var}^{2}-C.
\end{equation*}%
Noting that $Y_{m_{k}}^{n_{k}}$, $k\geq 1$, are independent, use
Borel-Cantelli lemma, 
\begin{eqnarray*}
&&P\left( \left\Vert X\right\Vert _{2-var}^{2}=\infty \right) \\
&\geq &P\left( \left\Vert Y\right\Vert _{2-var}^{2}=\infty \right) \geq
P\left( \overline{\lim }_{k\rightarrow \infty }\left\{ \left\Vert \frac{%
Y_{m_{k}}^{n_{k}}}{k}\right\Vert _{2-var}^{2}>1\right\} \right) =1.
\end{eqnarray*}
\end{proof}

In fact, the method above can be applied to all orthogonal systems in the
form $\left\{ \varphi \left( nx\right) \right\} _{n\geq 1}$, $x\in \left[ 0,1%
\right] $, where $\varphi $ is an $\alpha $-H\"{o}lder continuous function, $%
0<\alpha \leq 1$.

\section{Acknowledgments}

The research of both authors are supported by European Research Council
under the European Union's Seventh Framework Programme (FP7-IDEAS-ERC) / ERC
grant agreement nr. 291244. The research of Terry Lyons is supported by
EPSRC grant EP/H000100/1. The authors acknowledge the support of the
Oxford-Man Institute, and want to thank the referees for constructive
suggestions.

\end{document}